\documentclass{amsart}

\usepackage{graphicx}

\newtheorem{theorem}{Theorem}[section]
\newtheorem{lemma}[theorem]{Lemma}
\newtheorem{proposition}[theorem]{Proposition}

\theoremstyle{corollary}
\newtheorem{corollary}[theorem]{Corollary}

\theoremstyle{definition}

\theoremstyle{remark}
\newtheorem{remark}[theorem]{Remark}

\numberwithin{equation}{section}

\begin{document}

\title[Multi-orthogonality  and bulk queueing theory]{A multi-orthogonal polynomials' approach to bulk queueing theory}

\author{U. Fidalgo}
\address{Department of Mathematics and Statistics,
Case Western Reserve University, Cleveland, Ohio 43403}
\email{uxf6@case.edu}



\date{\today}

\dedicatory{In memoriam to CWRU's Professors Elizabeth Meckes and Wojbor A. Woyczynski}


\begin{abstract} We consider a stationary Markov process that models certain queues with a bulk service  of a fixed number $m$ of admitted customers. We find an integral expression of its transition probability function in terms of certain multi-orthogonal polynomials with respect to a system of distributions that contain measures supported on starlike subsets of the complex plane. We also study the corresponding steady state.
\end{abstract}

\maketitle                                        

\section{Introduction}\label{introduction}

In \cite{B} N. Bailey considered a queueing model with a bulk service admitting batches with a maximum size $m\in \mathbb{N}=\{1,2,3,\ldots\}$. This model describes several situations where the customers are admitted in groups, for example transportation of people in vehicles with a maximum passenger capacity.  In our case we assume that batches have a fixed size $m$ (see \cite[Chapter 4]{Med}). We have an example of these situations in some shuttle routes where a bus only starts its trip if it completes the vehicle's capacity.  We can find more examples of this behavior in certain bureaucratic processes, for instance this happens in some public museums where the guides only admit a fixed number of visitors and if this number is not completed they don't work. 

Let us assume that an individual arrives at each epoch of Poisson occurrence with rate $\lambda>0$, and  the service only starts if the queue size reaches or exceeds  $m$. The service for batches is also assumed a Poisson process with parameter $\mu>0$. In terms of Kendall's notation \cite{Ke} used in \cite{Med}, we say that our system is  $M/M(m,m)/1$. 

This system corresponds to a stationary Markov process with path function $X(t)$ whose transition probability function has form
\[
P_{i,j}(t)=Pr \left\{X(t+s)=j | X(s)=i\right\}, \quad (i,j) \in \mathbb{Z}_+\times \mathbb{Z}_+=\mathbb{Z}_+^2,
\]
where $\mathbb{Z}_+=\mathbb{N}\cup \{0\}=\{0,1,2,\ldots\}.$ 

\begin{remark}\label{fourconditions}
Consider the infinite matrix function 
\[
\displaystyle {\bf P}(t)=[P_{i,j}(t)]_{(i,j)\in \mathbb{Z}_+^2}. 
\]
From \cite[pages 5-7]{And} we have that $\displaystyle {\bf P}=[P_{i,j}]_{(i,j)\in \mathbb{Z}_+^2}$  is a transition probability function if it satisfies
\begin{enumerate}
\item[]
\item $P_{i,j}(t)\geq 0$, $t\in \mathbb{R}_+$, and $P_{i,j}(0)=\delta_{i,j}$ where
\[
\delta_{i,j}=\left\{\begin{array}{l l l}
0 & \mbox{if} & i\not = j, \\ & & \\
1 &\mbox{if} & i= j,
\end{array}
\right.
\]
denotes Kroneker's delta. We denote ${\bf P}(t) \geq 0$. Also ${\bf P}(0)=\mathbb{I}$. The last equality is the so called identity property.
\item[]
\item $\displaystyle \sum_{j=0}^{\infty} P_{i,j}(t) \equiv 1$, $t\in \mathbb{R}_+$, $n\in \mathbb{Z}_+$. This property is called honesty.
\item[]
\item ${\bf P}(s+t)={\bf P}(s){\bf P}(t)$, $(s,t)\in \mathbb{R}_+^2$, which is the semigroup property (Kolmogorov-Chapman equation).
\item[]
\item $\displaystyle \lim_{t\to 0}P_{i,j}(t)=\delta_{i,j}$, $(i,j)\in \mathbb{Z}_+^2$, or $\displaystyle \lim_{t\to 0}{\bf P}(t)=\mathbb{I}$.
\item[]
\end{enumerate}
According to the definitions in \cite[Chaptter 5]{And}, a set $\mathbb{E} \subset \mathbb{Z}_+$ is irreductible for a transition probability function $P_{i,j}(t)$ if every two states $i$ and $j$ in $\mathbb{E}$ can be reached from each other, symbolically  it is $P_{i,j}(t)>0$ for all $(i,j)\in \mathbb{E}^2$, with $t>0$.  We say that a state  $i \in \mathbb{Z}_+$ is transient when the improper integral
\begin{equation}\label{expectationn}
\int_0^{\infty} P_{i,i}(t) d \,t <\infty
\end{equation}
is convergent. The integral in (\ref{expectationn}) represents the expectation of  the time spent  by the random variable $X(t)=i$, if it  starts at $X(0)=i$. When the improper integral in (\ref{expectationn}) diverges to infinity, we say that $i$ is  persisten. In Theorem \ref{main} we state a transition probability function $P_{i,j}$, $(i,j)\in \mathbb{Z}_+^2$ that satisfies the four conditions in Remark \ref{fourconditions} where $\mathbb{Z}_+$ is irreductible and when $m>1$ every state is transient (there is no persistent state). This describes well the $M/M(m,m)/1$ queueing model that we are dealing where there is no priority state.

In our model $M/M(m,m)/1$ the following conditions hold
\begin{equation}\label{keycond}
\begin{array}{l}
\left\{\begin{array}{rcl}
P_{i,i}(\Delta t)&=&1-\lambda\Delta t+o(\Delta t), \\ & & \\
P_{i,i+1}(\Delta t)&=& \lambda \Delta t+o(\Delta t),
\end{array} \quad \mbox{as} \quad \Delta t \to 0, \quad \mbox{if} \quad i < m.\right.\\\\
\left\{\begin{array}{rcl}
P_{i,i-m}(\Delta t)&=&\mu \Delta t + o (\Delta t), \\ & & \\
P_{i,i}(\Delta t)&=&1-(\lambda+\mu)\Delta t+o(\Delta t), \\ & & \\
P_{i,i+1}(\Delta t)&=& \lambda \Delta t+o(\Delta t).
\end{array} \quad \mbox{as} \quad \Delta t \to 0, \quad \mbox{if} \quad i \geq m.\right.
\end{array}
\end{equation}
\end{remark}

When $m=1$ we are in a case of a Birth-and-Death process (see \cite{KMc}).  Birth and Death processes have been used to model any M/M(1,1)/1 system  (see \cite[Chapter III]{AB}). As in  \cite[Chapter XVII]{Fe}), considering (\ref{keycond}) with the Chapman-Kolmogorov equations, we obtain the following system with infinitely many differential equations. Set $P_{0,-1}=0$
\begin{equation}\label{recurrencePro}
P_{i,j}^{\prime} (t)=\left\{\begin{array}{l}
-\lambda P_{i,j}(t)+\lambda P_{i+1,j}(t), \quad (i,j) \in   \left\{0,1,\ldots, m-1\right\} \times \mathbb{Z}_+=\mathbb{F},\\  \\
-(\lambda+\mu)P_{i,j}(t)+\lambda P_{i+1, j}(t)+ \mu P_{i-m, j}(t), \quad (i,j) \in \mathbb{Z}_+^2 \setminus \mathbb{F},
\end{array}
\right.
\end{equation}
with $t \geq 0$ (we use the notation $t\in \mathbb{R}_+$).  

Combining condition $(1)$ in Remark \ref{fourconditions} with the system of differential equations in (\ref{recurrencePro}) we have an initial value problem that we write in its matrix version as follows
\begin{equation}\label{initialvalueproblem}
\left\{
\begin{array}{l l}
{\bf P}^{\prime}(t)= {\bf A} {\bf P}(t), & t\geq 0,  \\ & \\
{\bf P}(0)=\mathbb{I},
\end{array}
\right.
\end{equation}
where ${\bf A}=[a_{i,j}]_{(i,j)\in \mathbb{Z}_+^2}$ is such that
\[
a_{i,j}=\left\{\begin{array}{r c l}
\mu & \mbox{if} & j=i-m, \\ & & \\
-\lambda & \mbox{if} & j=i<m, \\ & & \\
-(\lambda+\mu) & \mbox{if} & j=i\geq m, \\ & & \\
\lambda & \mbox{if} & j=i+1, \\ & & \\
0 &\mbox{if} & j\not \in \{i+1,i,i-m\},
\end{array}\right. \quad (i,j) \in \mathbb{Z}_+^2,
\]
which has schematically form

\begin{equation}\label{Amatrix}
{\bf A}=\left(
\begin{array}{c c c c c c c c c}
 -\lambda & \lambda   &            0 & \cdots   & 0               & 0                     &  0           & 0          & \cdots  \\
     0         & -\lambda  & \lambda &              & 0               & 0                     &  0           & 0          &             \\ 
\vdots      & \ddots      & \ddots    & \ddots   & \vdots       &                         &  \vdots    & \vdots   &             \\ & & & & & & & & \\
0              & 0             & \cdots     & 0           & -\lambda  & \lambda            & 0            & 0           & \cdots   \\  
\mu          & 0             & \cdots     &  0          & 0              & -(\lambda+\mu) & \lambda &  0          & \ddots  \\ 
0              &\mu          &                &   0        &  0              & 0                       & -(\lambda+\mu)      & \lambda    &   \ddots   \\ 
\vdots      & \ddots     &  \ddots     &             & \ddots      &                           &  \ddots                    & \ddots       & \ddots  \\          
\end{array}
\right).
\end{equation}

The system of differential equations (\ref{recurrencePro}) is deduced from (\ref{keycond}), which is independent from  the four conditions in  Remark \ref{fourconditions}, hence some of its solutions would not be transition probability functions. In order to find a matrix function ${\bf P}$ that describes  properly a system $M/M(m,m)/1$, we select from the set of solutions of (\ref{initialvalueproblem}) the ones that satisfy such four conditions. From the publication \cite{S} by W. Sternberg, we deduce that there is a unique solution with analytic entries functions on the complex plane $\mathbb{C}$, being real valued on $\mathbb{R}_+=[0, \infty)$, and so that it satisfies the four conditions that define what a transition probability function is. The other possible solutions have non uniform bounded entries on $\mathbb{R}_+$. In case such solutions exist, conditions (1) and (2) in Remark  \ref{fourconditions} would not be satisfied with them, hence they would not be transition probability functions.

In \cite{KMc}, S. Karlin and J. McGregor arrived at an integral expression for the entries of ${\bf P}$ in the case of $m=1$, which is a Birth-and-Death process. They used some aspects of the orthogonal polynomial theory. Using Karlin and McGregor's approach,  T. S. Chihara and M. E. H. Ismail (see  \cite{ChIs}) connected sequences of orthogonal polynomials with a queueing model. Later on in \cite{Is} M. E. H. Ismail revisited this connection.  In this paper we also use Karlin and McGregor's method to give an accurate integral expressions for the mentioned entries of the transition probability function which is a solution of the initial value problem (\ref{initialvalueproblem}), in our case for general $m \geq 1$. When $m=1$ (Karlin \& Mc Gregor's case) the matrix ${\bf A}$ in (\ref{Amatrix}) is symmetrizable and  it induces a symmetric operator in the Hilbert space of sequences with the usual inner product (dot product). In our situation of a queueing model with coefficient constant with respect to time, this operator, that we can also call ${\bf A}$,  is a compact selfadjoint operator and there is a spectral measure supported on the ${\bf A}$'s spectrum which is contained in an interval of the real line. So when ${\bf A}$ is compact Karlin and Mc Gregor's result could be expressed in terms of the orthogonal polynomials family associated to the ${\bf A}$'s spectral measure. When $m>1$ the corresponding operator ${\bf A}$ is not selfadjoint, and extending these ideas is not straightforward. Actually this is the principal difficulty in the present work. To avoid such inconvenience we use a family of multi-orthogonal polynomials defined by a higher order three terms of recurrence relation, which is a variation of others studied in \cite{AKS,AKV,AL1,AL2, AE}. 

Let $\displaystyle \left\{Q_n\right\}_{n\in \mathbb{Z}_+}$ be a sequence of polynomials generated by the recurrence relation
\begin{equation}\label{recurrenceQ}
(\lambda+\mu+x)Q_n(x)=\lambda Q_{n+1}(x) +\mu Q_{n-m}(x), \quad n \geq m,  
\end{equation}
with initial conditions
\begin{equation}\label{recurrenceQv}
Q_n(x)=\frac{1}{\lambda^n}(\lambda+x)^{n},\quad n=0,\ldots, m.
\end{equation}
The conditions above (\ref{recurrenceQv}) can also be written as 
\begin{equation}\label{recurrenceQv2}
Q_0\equiv 1, \quad Q_n(x)=\frac{1}{\lambda} (\lambda+x) Q_{n-1},\quad n=1,\ldots, m.
\end{equation}
We consider a dual family of vector polynomials $\displaystyle \left\{{\bf q}_r=(q_{0,r},\ldots,q_{m-1,r})\right\}_{r \in \mathbb{Z}_+}$ whose elements satisfy
\begin{equation}\label{preccurencelatin}
\begin{array}{l}
\displaystyle {\bf q}_{-1}=(0,\ldots,0), {\bf q}_{0}=\left(1,0,\ldots,0\right),  \\ \\
{\bf q}_{1}=\left(0,1,0,\ldots,0\right), \ldots, {\bf q}_{m-1}=\left(0,\ldots,, 0,1\right),  \\ \\
(\lambda+x) {\bf q}_r(x)=\mu{\bf q}_{r+m}(x)+\lambda {\bf q}_{r-1}(x), \quad  r \in \{0,\ldots, m-1\},\\ \\
(\lambda+\mu+x) {\bf q}_r(x)=\mu{\bf q}_{r+m}(x)+\lambda {\bf q}_{r-1}(x), \quad  r \in \{m,m+1,\ldots,\}.
\end{array}
\end{equation}

Given two complex numbers $u$ and $v$ we denote by $[u,\, v]$ the compact segment of straight line that connects $u$ with $v$. A real valued function $r$ defined on a certain domain $D\subset \overline{\mathbb{C}}$ of the Riemann sphere $\overline{\mathbb{C}}=\mathbb{C} \cup \{\infty\}$ is said to be a weight if it does not change sign on $D$. This is $r: D \subset \overline{\mathbb{C}} \rightarrow \mathbb{R}_+$ or $r: D \subset \overline{\mathbb{C}} \rightarrow \mathbb{R}_-=(\mathbb{R}\setminus \mathbb{R}_+)\cup \{0\}$.  Let $\delta_{\zeta}$ denote the measure often called Riemann's delta with an only  point of mass at  $\zeta \in \mathbb{R}$. Fix $j\in \mathbb{Z}_+$. According to \cite[Chapter 6]{Rud} $\delta_{\zeta}^{(j)}$ denotes the $j$th distribution derivative of $\delta_{\zeta}$, where for each $j$th differentiable function $f$ at $\zeta$ whose derivatives $f^{(k)}$, $k=0,1,\ldots,j$ are integrable in $\mathbb{R}$, satisfies that
\begin{equation}\label{deltaderivative}
\int f(z) \, \delta_{\zeta}^{(j)}(z)\,  d z=(-1)^{j}f^{(j)}(\zeta).
\end{equation}
We can extend  (\ref{deltaderivative}) for any differentiable function $f$ in an open interval that contains  $\zeta$, replacing $f$ by a proper test function expressed by a piecewise definition which coincides with $f$ on such interval. Fixed $j\in \mathbb{N}$ the distribution $\delta_{\zeta}^{(j)}$ is not a measure. The following result stated in Theorem \ref{multiQ} is proved in Section \ref{thproofs}, specifically in its Subsection \ref{HPD}.

\begin{theorem}\label{multiQ} There exists a system of distribution $(\sigma_0,\sigma_1, \ldots, \sigma_{m-1})$ such that the following bi-orthogonality relations hold 
\begin{equation}\label{orthogonalityQ}
 \delta_{n,r}=\int Q_n(x) \sum_{j=0}^{m-1} q_{j,r}(x) \, d \sigma_j(x), \quad (n,r)\in \mathbb{Z}_+^2.
 \end{equation}
Each distribution component $\sigma_j$, $j=0,1,\ldots,m-1$, has a type of Radon Nikodym derivative
\begin{equation}\label{RNdistribution}
\frac{d \sigma_j}{d x}=\rho_j (x) \rho(x)+\frac{(-1)^{j-1}}{\lambda^{j}(j-1)!}\left(\delta_{-\lambda}-\delta_{-\lambda-\mu}\right)^{(j-1)}, \quad j=0,1,\ldots, m-1
\end{equation}
with $\rho_0 \equiv 1$ and $\displaystyle \frac{\delta^{(-1)}}{(-1)!}=0$. The expression $\rho_j (x) \rho(x)$ is  a function defined on the starlike set
\[
\Sigma_0=\bigcup_{k=0}^m \left[-\lambda-\mu, -\lambda-\mu + \frac{m+1}{m}\left(\frac{\mu\lambda^m}{m}\right)^{1/(m+1)} \exp \frac{2 \pi i k}{m+1}\right],
\]
 where fixed $k \in \{0,1,\ldots, m\}$,  for each point 
\[
x_k=-\mu-\lambda+ t \exp \frac{2 \pi i k}{m+1} \in \Sigma_0, \quad 0\leq t \leq  \frac{m+1}{m}\left(\frac{\mu\lambda^m}{m}\right)^{\frac{1}{m+1}},
\]
the following symmetry holds
\[
\rho_j \left(x_k\right) \rho \left(x_k\right)= \exp \left(-\frac{2 \pi i j}{m+1} \right)\rho_j \left(t-\mu-\lambda\right) \rho \left(t-\mu-\lambda\right),
\]
being $\displaystyle \rho_j \left(t-\mu-\lambda\right) \rho \left(t-\mu-\lambda\right)$ a weight.
\end{theorem}

The functions $\rho, \rho_1, \ldots, \rho_{m-1}$ can be expressed explicitly in terms of the solutions of the equation 
\begin{equation}\label{omegaequationQ}
\lambda\, \omega^{m+1}-(z+\lambda+\mu)\, \omega^m+\mu=0.
\end{equation}

Let us state our main result

\begin{theorem}\label{main} The matrix function ${\bf P}(t)=[P_{n,r}]_{(n,r)\in \mathbb{Z}_+^2}$ whose entries are
\begin{equation}\label{Abbey}
P_{n,r}(t)=\int e^{x t} Q_n(x) \sum_{j=0}^{m-1} q_{j,r}(x) \, d \sigma_j(x),  \qquad (n,r)\in \mathbb{Z}_+^2,
\end{equation}
is the unique solution of the initial value problem (\ref{initialvalueproblem}), that conform a transition probability function. In addition for $P_{n,r}$ the sets $\mathbb{Z}_+$ is irreductible and if $m>1$ any state is transient with $P_{n,r}$ tending to the trivial steady state 
\begin{equation}\label{steady}
\displaystyle \lim_{t \to 0} P_{n,r}(t)=0, \qquad \mbox{exponentially.}
\end{equation}.
\end{theorem}

We prove Theorem \ref{multiQ} and Theorem \ref{main} in Section \ref{thproofs}, in its subsections \ref{HPD} and \ref{remark0}, respectively. The proof of Theorem \ref{main} results after showing that the functions $P_{n,r}$, $(n,r)\in \mathbb{Z}_+^2$  in (\ref{Abbey}) conform a solution corresponding to the initial value problem in (\ref{initialvalueproblem}), and also satisfy the properties corresponding transition probability functions  enumerated in  Remark \ref{fourconditions}. In the proof of Theorem \ref{multiQ} we used some results that are stated in Section \ref{Favard}. Some of these statements are already published in journals which are specialized in approximation theory or orthogonal polynomials (for example \cite{AKS,AKV}). Since the reader of these journals are familiar with these topics, the author perhaps did not feel the need to go into the details of their proofs. Some of these details are necessary to complete the proof of Theorem \ref{multiQ}. Considering that this work connects different math branches, we show some of these previous results in Section \ref{Favard}, adapting the generality of the statements to the present context and notation.

\section{Previous necessary results}\label{Favard}

In this Section we follow ideas from \cite{AKS, AKV} to arrive at necessary results to prove Theorem \ref{multiQ} and Theorem \ref{main} in Section \ref{thproofs}. 

\subsection{linear operators and properties}

From the family of polynomials $\displaystyle \left\{Q_n\right\}_{n\in \mathbb{Z}_+}$ defined by the recurrence relation (\ref{recurrenceQ}) and initial conditions (\ref{recurrenceQv}), we introduce another $\displaystyle \left\{L_n\right\}_{n\in \mathbb{Z}_+}$ the change of variable $z=x+\lambda+\mu$ with
\[
L_n(z)= \lambda^n Q_n(z-\lambda-\mu), \quad n\in \mathbb{Z}_+.
\]
Taking into account the relations (\ref{recurrenceQ}) and (\ref{recurrenceQv}) we have that
\begin{equation}\label{recurrenceK}
z\, L_n(z)=L_{n+1}(z)+\mu\lambda^m L_{n-m}, \quad n\geq m,
\end{equation}
with initial conditions
\begin{equation}\label{recurrenceKI}
L_n(z)=(z-\mu)^n, \quad n=0,\ldots,m-1.
\end{equation}
We also introduce a third family of polynomials $\displaystyle \left\{T_n\right\}_{n\in \mathbb{Z}_+}$ satisfying the recurrence relation (\ref{recurrenceK}) 
\begin{equation}\label{recurrenceT}
z\, T_n(z)=T_{n+1}(z)+\mu\lambda^m T_{n-m}, \quad n\geq m,
\end{equation}
with different initial conditions
\begin{equation}\label{recurrenceTI}
T_n(z)=z^n, \quad n=0,\ldots,m-1.
\end{equation}
There is an equivalent way to define the sequence of polynomials $\displaystyle \left\{T_n\right\}_{n\in \mathbb{Z}_+}$:
\begin{equation}\label{recurrenceTN}
z\, T_n(z)=T_{n+1}(z)+\mu\lambda^m T_{n-m}, \quad n\in \mathbb{Z}_+,
\end{equation}
with different initial conditions
\begin{equation}\label{recurrenceTIN}
T_{-m}=T_{-m+1}=\cdots=T_{-1}=0, \quad \mbox{and} \quad T_0=1.
\end{equation}
Sequences of polynomials with recurrence relations of $\displaystyle \{T_n\}_{n\in \mathbb{Z}_+}$  are analyzed in several publications such as \cite{AKS, AKV, AL1, AL2, AE}. Section \ref{classical} deals specifically about these polynomials. These two sequences $\displaystyle \left\{L_n\right\}_{n\in \mathbb{Z}_+}$ and $\displaystyle \left\{T_n\right\}_{n\in \mathbb{Z}_+}$ are used in Section \ref{HPD}.

As we mentioned in the introduction we consider the linear operator ${\bf A}$ acting in the space of sequences $\ell^2$ determined by the infinite non-symmetric matrix ${\bf A}$ in (\ref{Amatrix}), which is associated to the recurrence relation (\ref{recurrenceQ}) of the sequence $\displaystyle \left\{Q_n\right\}_{n\in \mathbb{Z}_+}$, also determined by the initial conditions (\ref{recurrenceQv}). The linear operator ${\bf A}$  is defined by
\[
{\bf A}^{\top} {\bf e}_j=\lambda {\bf e}_{j+1}-\lambda {\bf e}_j \quad j\in \{0,1,2, \ldots, m-1\}
\]
and
\[
{\bf A}^{\top} {\bf e}_j=\lambda {\bf e}_{j+1}-(\lambda+\mu) {\bf e}_j+\mu  {\bf e}_{j-m}, \quad j\geq m.
\]
Analogously  we introduce two more linear operators ${\bf L}$ and ${\bf T}$ corresponding to the recurrence relations of $\displaystyle \{L_n\}_{n\in \mathbb{Z}_+}$ and $\displaystyle \{T_n\}_{n\in \mathbb{Z}_+}$, respectively, as follows:
\begin{equation}\label{operatorT}
{\bf L}^{\top} {\bf e}_j={\bf e}_{j+1}-\mu {\bf e}_j, \quad 0\leq j\leq m-1, \quad {\bf L}^{\top} {\bf e}_j={\bf e}_{j+1}+\mu\lambda^m  {\bf e}_{j-m}, \quad j\geq m,
\end{equation}
and 
\[
{\bf T}^{\top} {\bf e}_j={\bf e}_{j+1}, \quad  0\leq j\leq m-1, \quad {\bf T}^{\top} {\bf e}_j={\bf e}_{j+1}+\mu\lambda^m  {\bf e}_{j-m}, \quad j\geq m.
\]

\begin{remark}\label{matrixidentityremark} The matrix expressions of the operators ${\bf A}$ and ${\bf L}$ are related by the expression 
\begin{equation}\label{matrixidentity}
{\bf A}={\bf \Lambda}^{-1} \left[{\bf L}-(\lambda+\mu)\mathbb{I}\right] {\bf \Lambda},
\end{equation}
where
\[
{\bf \Lambda}= \left(\begin{array}{c c c c}
1 & 0 & 0 & \cdots \\  
0 & \lambda & 0 &  \\ 
0 & 0               & \lambda^2 & \ddots \\ 
\vdots &  & \ddots & \ddots
\end{array}
\right).
\]
This relationship between ${\bf L}$ and ${\bf A}$ is used in Section \ref{HPD}.
\end{remark}

Since the three sequences of polynomials    $\displaystyle \{L_n\}_{n\in \mathbb{Z}_+}$, $\displaystyle \{L_n\}_{n\in \mathbb{Z}_+}$ and $\displaystyle \{T_n\}_{n\in \mathbb{Z}_+}$ have similar properties, we summarize their recurrence relations (\ref{recurrenceQ}), (\ref{recurrenceK}, and (\ref{recurrenceT}), and their respective initial conditions (\ref{recurrenceQv}),  (\ref{recurrenceKI}), and  (\ref{recurrenceTI}), with the ones corresponding to a family of polynomials $\displaystyle \left\{H_n\right\}_{n\in \mathbb{Z}_+}$. Set $\eta>0$, $\iota >0$, $\xi \in \mathbb{R}_+$, and $\gamma \in \mathbb{R}$:
\begin{equation}\label{recurrenceH}
(\xi+x)H_n(x)=\iota H_{n+1}(x) +\eta H_{n-m}(x), \quad n \geq m,
\end{equation}
with initial conditions
\begin{equation}\label{recurrenceHI}
H_n(x)=\frac{1}{\iota^n}(\gamma+x)^{n},\quad n=0,\ldots, m.
\end{equation}
We also consider the associated  linear operator ${\bf H}$ acting on the space of sequences $\ell^2$, which is determined by a infinite non-symmetric matrix $\displaystyle {\bf H}=[h_{i,j}]_{(i,j)\in \mathbb{Z}_+^2}$ with entries satisfying
\[
h_{i,j}=\left\{\begin{array}{r c l}
\eta & \mbox{if} & j=i-m, \\ & & \\
-\gamma & \mbox{if} & j=i<m, \\ & & \\
-\xi & \mbox{if} & j=i\geq m, \\ & & \\
\iota & \mbox{if} & j=i+1, \\ & & \\
0 &\mbox{if} & j\not \in \{i+1,i,i-m\},
\end{array}\right. \quad (i,j) \in \mathbb{Z}_+^2,
\]
with schematic form
\begin{equation}\label{Hmatrix}
{\bf H}=\left(
\begin{array}{c c c c c c c c c}
 -\gamma & \iota   &            0 & \cdots   & 0               & 0                     &  0           & 0          & \cdots  \\
     0         & -\gamma  & \iota &              & 0               & 0                     &  0           & 0          &             \\ 
\vdots      & \ddots      & \ddots    & \ddots   & \vdots       &                         &  \vdots    & \vdots   &             \\ & & & & & & & & \\
0              & 0             & \cdots     & 0           & -\gamma  & \iota            & 0            & 0           & \cdots   \\  
\eta          & 0             & \cdots     &  0          & 0              & -\xi & \iota &  0          & \ddots  \\ 
0              &\eta          &                &   0        &  0              & 0                       & -\xi      & \iota    &   \ddots   \\ 
\vdots      & \ddots     &  \ddots     &             & \ddots      &                           &  \ddots                    & \ddots       & \ddots  \\          
\end{array}
\right).
\end{equation}
 Set $\{{\bf e}_j\}_{j\in \mathbb{Z}_+}$ the standard basis of $\ell^2$, then the operator ${\bf H}$ is defined as follows
\begin{equation}\label{operadorH}
{\bf H}^{\top} {\bf e}_j=\left\{
\begin{array}{l l}
\iota {\bf e}_{j+1}-\gamma {\bf e}_j ,& j\in \{0,1,2, \ldots, m-1\}, \\ & \\
\iota {\bf e}_{j+1}-\xi {\bf e}_j+\eta  {\bf e}_{j-m}, & j\geq m,
\end{array}\right.
\end{equation}
where $\top$ denotes the operator of matrix transposition. 

Given a polynomial $p (x)=b_{\kappa}\,x^{\kappa}+\cdots+b_1\,x+b_0$ with an arbitrary degree $\kappa$ we denote the operator $p({\bf H})$ satisfying that
\[
p({\bf H}^{\top}){\bf u}=\left(b_{\kappa} {\bf H}^{\top \kappa}+\cdots+b_1\, {\bf H}^{\top}+ b_0\right){\bf u}=b_{\kappa} {\bf H}^{\top \kappa}{\bf u}+\cdots+b_1\, {\bf H}^{\top}{\bf u}+ b_0{\bf u}, \quad {\bf u}\in \ell^2.
\]   

\begin{lemma}\label{jumplemma} Consider the sequence of polynomials $\displaystyle \left\{H_n\right\}_{n\in \mathbb{Z}_+}$ satisfying the recurrence relation (\ref{recurrenceH}). Then
\begin{equation}\label{comiteH}
{\bf e}_n=H_n({\bf H}^{\top}) \, {\bf e}_0, \quad n\in \mathbb{Z}_+.
\end{equation}
\end{lemma}

\begin{proof} For each $\displaystyle n \in \left\{0,1,\ldots, m-1\right\}$ we have that $\displaystyle H_n(x)=\frac{1}{\iota^n}(\gamma+x)^n$, then $\displaystyle H_n({\bf H}^{\top})=\frac{1}{\iota^n}({\bf H}^{\top}+\gamma)^n$. This means that 
\[
\begin{array}{r c l}
H_0({\bf H}^{\top})\,{\bf e}_0 &=& \displaystyle \frac{1}{\iota^0}\left({\bf H}^{\top} +\gamma \right)^0 {\bf e}_0={\bf e}_0, \\ & & \\
H_1({\bf H}^{\top})\,{\bf e}_0 &=& \displaystyle \frac{1}{\iota}\left({\bf H}^{\top} +\gamma \right) {\bf e}_0=\frac{1}{\iota}\left(\iota {\bf e}_1-\gamma{\bf e}_0+\gamma {\bf e}_0\right)={\bf e}_1,\\ & & \\
H_2({\bf H}^{\top})\, {\bf e}_0&=&\displaystyle \frac{1}{\iota^2}({\bf H}^{\top}+\gamma)^2  {\bf e}_0=\frac{1}{\iota}\left({\bf H}^{\top}+\gamma\right) {\bf e}_1=\frac{1}{\iota}\left(\iota {\bf e}_2-\gamma {\bf e}_1+\gamma {\bf e}_1\right)={\bf e}_2,\\ & & \\
                                             &\vdots& \\& & \\
H_{m-1}({\bf H}^{\top})\, {\bf e}_0&=&\displaystyle \frac{1}{\iota^{m-1}}({\bf H}^{\top}+\gamma)^{m-1}  {\bf e}_0\\ & & \\
&=& \displaystyle \frac{1}{\iota^{m-2}}({\bf H}^{\top}+\gamma)^{m-2}  {\bf e}_1=\cdots= \frac{1}{\iota}\left({\bf H}^{\top}+\gamma\right)  {\bf e}_{m-2}={\bf e}_{m-1}.                                      
\end{array}
\]
We now consider $n\geq m-1$. For each $k=1,\ldots,n$, assume that $H_k({\bf H}^{\top})\,{\bf e}_0={\bf e}_k$. By the recurrence relation (\ref{recurrenceH}), we have that
\[
H_{n+1}({\bf H}^{\top}){\bf e}_0= \frac{1}{\iota}\left[({\bf H}^{\top} +\xi) H_{n}({\bf H}^{\top})-  \eta H_{n-m}({\bf H}^{\top})\right]{\bf e}_0.
\]
hence
\[
H_{n+1}({\bf H}^{\top})=\frac{1}{\iota}\left[\iota {\bf e}_{n+1}-\xi{\bf e}_n+\xi{\bf e}_n+\eta   {\bf e}_{j-m}-\eta {\bf e}_{n-m}\right]={\bf e}_{n+1}.
\]
This completes the proof by induction.
\end{proof}

Let us writeWe only need to substitute 

\begin{remark}\label{jumplemmaQLT}  Substituting the following relations in Lemma \ref{jumplemma}
\begin{enumerate}
\item[]
\item for $\displaystyle \{Q_n\}_{n \in \mathbb{Z}_+}$, $\gamma=\iota=\lambda$, $\eta=\mu$, and $\xi=\lambda+\mu$, 
\item[]
\item for $\displaystyle \{L_n\}_{n \in \mathbb{Z}_+}$, $\gamma=-\mu$, $\iota=1$, $\eta=\mu \lambda^m$, and $\xi=0$, 
\item[]
\item For $\displaystyle \{T_n\}_{n \in \mathbb{Z}_+}$, $\gamma=\xi=0$, $\iota=1$, and $\eta=\mu \lambda^m$,
\item[]
\end{enumerate}
we are able to give the statement of  Lemma \ref{jumplemma} for the three operators ${\bf A}$, ${\bf L}$, and ${\bf T}$:
\begin{equation}\label{comite}
{\bf e}_n=Q_n({\bf A}^{\top}) \, {\bf e}_0=L_n({\bf L}^{\top}) \, {\bf e}_0=H_n({\bf H}^{\top}) \, {\bf e}_0, \quad n\in \mathbb{Z}_+.
\end{equation}
\end{remark}

The following results Lemma \ref{tiberio} and Lemma \ref{canchero} refers the cases of  ${\bf L}$ and ${\bf T}$. So we call ${\bf K}$ a linear operator belonging to the class of the ${\bf H}$'s where  $\eta=\mu\lambda^m$, $\xi=0$, and $\iota=1$, while $\gamma$ takes the values either $\gamma=\mu$ for ${\bf L}$ or $\gamma=0$ for ${\bf T}$. We also denote the sequence of polynomials $\displaystyle \left\{K_n\right\}_{n\in \mathbb{Z}_+}$ obtained by the recurrence relation (\ref{recurrenceH}) with the initial conditions (\ref{recurrenceHI}) with these values for $\eta,$ $\xi$, $\iota$, and $\gamma$.

\begin{lemma}\label{tiberio} Set $\displaystyle j \in \left\{1,\ldots,m\right\}$. For each $\displaystyle \nu \in  \{0,1, \ldots, j-1\}$
\begin{equation}\label{gens}
 \displaystyle ({\bf K}+\gamma\mathbb{I})^{\nu} {\bf e}_{j-1}=\sum_{k=0}^{\nu}{\nu \choose k} \left(\mu \lambda^m\right)^{\nu-k} {\bf e}_{(\nu-k)m+j-1-k}.
\end{equation}
\end{lemma}

\begin{proof} We proceed by induction. In case $\nu=0$ according to (\ref{operadorH}) we have that
\[
({\bf K}+\gamma\mathbb{I})^0 {\bf e}_{j-1}={\bf e}_{j-1}.
\]
Then assuming that the statement (\ref{gens}) is true for a $\nu \in \{0,1,\ldots j-2\}$ we prove it for $\nu+1$. 
\[
 \displaystyle ({\bf K}+\gamma\mathbb{I})^{\nu+1} {\bf e}_{j-1}=({\bf K}+\gamma\mathbb{I})({\bf K}+\mu\mathbb{I})^{\nu} {\bf e}_{j-1}
 \]
 \[
 =({\bf K}+\gamma \mathbb{I}) \sum_{k=0}^{\nu}{\nu \choose k} \left(\lambda \mu^m\right)^{\nu-k} {\bf e}_{(\nu-k)m+j-1-k}
 \]
 \[
 =\sum_{k=0}^{\nu}{\nu \choose k} \left(\mu \lambda^m\right)^{\nu-k} ({\bf K}+\iota\mathbb{I})\, {\bf e}_{(\nu-k)m+j-1-k}
 \]
 \[
 \displaystyle = \sum_{k=0}^{\nu}{\nu \choose k} \left(\mu \lambda^m\right)^{\nu-k} \left({\bf e}_{(\nu-k)m +j-1-k-1}+\mu  \lambda^m{\bf e}_{(\nu-k+1)m+j-1-k}\right)
 \]
 \[
 \displaystyle = \sum_{k=0}^{\nu}{\nu \choose k} \left(\mu \lambda^m\right)^{\nu-k} {\bf e}_{(\nu-k)m +j-1-k-1}+ \sum_{k=0}^{\nu}{\nu \choose k} \left(\mu \lambda^m\right)^{\nu-k+1} {\bf e}_{(\nu-k+1)m+j-1-k}
 \]
 \[
 \displaystyle = \sum_{k=1}^{\nu+1}{\nu \choose k-1} \left(\mu \lambda^m\right)^{\nu-k+1} {\bf e}_{(\nu-k+1)m +j-1-k}
 \]
 \[
 + \sum_{k=0}^{\nu}{\nu \choose k} \left(\mu \lambda^m\right)^{\nu-k+1} {\bf e}_{(\nu-k+1)m+j-1-k}
 \]
 \[
  \displaystyle ={\bf e}_{j-1-\nu-1}+ \sum_{k=1}^{\nu}\left({\nu \choose k-1} +  {\nu \choose k}\right) \left(\lambda \mu^m\right)^{\nu-k+1} {\bf e}_{(\nu-k+1)m+j-1-k} 
  \]
  \[
  +\left(\lambda \mu^m\right)^{\nu+1}{\bf e}_{j-1-\nu-1}
  \]
  Now using Pascal's identity we obtain that
  \[
 \displaystyle ({\bf K}+\gamma \mathbb{I})^{\nu+1} {\bf e}_{j-1}=
 \]
  \[
   \displaystyle ={\bf e}_{j-1-\nu-1}+ \sum_{k=1}^{\nu}{\nu+1 \choose k}  \left(\mu \lambda^m\right)^{\nu-k+1} {\bf e}_{(\nu-k+1)m+j-1-k}+\left(\mu \lambda^m\right)^{\nu+1}{\bf e}_{j-1-\nu-1}
 \]
 \[
 = \sum_{k=0}^{\nu+1}{\nu +1\choose k} \left(\mu \lambda^m\right)^{\nu+1-k} {\bf e}_{(\nu+1-k)m+j-1-k}.
 \]
This completes the proof by induction.
\end{proof}

\begin{lemma}\label{canchero} The linear operator ${\bf K}$ has a bounded spectrum $\sigma_{\bf K}$ with
\[
\sigma_{\bf K}\subset \left\{z\in \mathbb{C}: |z|\leq 1+\mu \lambda^m \right\}\cup \left\{z\in \mathbb{C}: |z+\gamma|\leq 1+\mu \lambda^m \right\}.
\]
\end{lemma}

\begin{proof}
Set ${\bf K}_n=(K_{0}, \ldots, K_{n-1})^{\top}$, the vector ${\bf e}_{n}=(0,0\ldots,0,1)^{\top}$ with $n$ components, and the $n\times n$ matrix block $\displaystyle {\bf K}_n=[k_{i,j}]_{(i,j)\in \{0,1,\ldots, n-1\}^2}$ of the infinite matrix form corresponding to ${\bf K}$, where for each $i\in \{1,\ldots, n-1\}$
\[
k_{i, i+1}=1, \quad k_{i+m,i}=\mu \lambda^{m}, \quad k_{i,i}=-\gamma, \quad i\in \{0,\ldots,m\}\quad \mbox{otherwise} \quad k_{i,j}=0,
\]
schematically
\[
{\bf K}_n^{\top}=\left(
\begin{array}{c c c c c c c c c c c c }
-\gamma & 0  & 0  &\cdots  & 0 &\mu \lambda^m & 0        & 0  & \cdots &0 & 0 \\ 
1 & -\gamma          & 0  & \cdots  & 0 & 0      & \mu \lambda^m  & 0  & \cdots & 0 & 0 \\
0 & 1          & -\gamma  & \cdots  & 0 & 0      &  0  & \mu \lambda^m  &     \ddots     &0  & 0 \\
0 & 0          & 1 & \cdots  & 0 & 0      &  0  & 0  &           & 0 & 0 \\
\vdots  &     &  \ddots  &             &  \vdots  &         &  \vdots        &          &    \ddots    &  \ddots & \\
0& 0          & \cdots  & 1  & -\gamma & 0      &  0  & 0  &  & 0 & \mu \lambda^m \\
0& 0          & \cdots  & 0  & 1 & 0      &  0  & 0  & \cdots & 0 & 0  & \\
0& 0          & \cdots  & 0  & 0 & 1      &  0  & 0  & \cdots & 0 & 0   \\
\vdots  &     &        &  \vdots       & \ddots    &      \ddots   &    \ddots       &     \ddots     &          & \vdots & \vdots \\
0& 0          & \cdots  & 0   &\cdots  &    0   & 0  &  1 & 0 & 0 & 0 \\
0& 0          & \cdots  &  0 &\cdots  &      0 &  0 &  0 & 1 & 0 & 0 \\
0& 0          & \cdots  & 0  &\cdots  &      0 &  0&  0 & 0 & 1 & 0 \\
\end{array}
\right).
\] 
From the recurrence relations in (\ref{recurrenceK}) and  (\ref{recurrenceKI}) we observe that $\deg K_n=n$ then $K_n$ has $n$ zeros. Let $x_k$ denote the zeros of $K_n$. The relations (\ref{recurrenceK}) and  (\ref{recurrenceKI}) can be rewrite as follows
\[
z {\bf K}_n(z)={\bf K}_n\, (z)+K_{n}(z){\bf e}_n \implies x_k {\bf K}_n(x_k)={\bf K}_n\, (x_k),
\]
which implies that $x_k$, $k=1,\ldots, n,$ are the $n$ eigenvalues of ${\bf K}_n$. Now using  Gershgorin Disc Theorem (see \cite{Var}) we obtain that $|x_j+\gamma| \leq 1+\mu \lambda^m$ or $|x_j| \leq 1+\mu \lambda^m$, $j=1,\ldots,n$. Then the operator ${\bf K}$ has a bounded support $\mbox{supp}({\bf K})\subset \left\{z\in \mathbb{C}: |z|\leq 1+\mu \lambda^m \right\}\cup \left\{z\in \mathbb{C}: |z+\gamma|\leq 1+\mu \lambda^m \right\}$. This completes the proof.
\end{proof}

Given a linear operator with the form of ${\bf H}$, we consider its $m$ resolvent functions (see for instance \cite{Ka,NiSo}), which are one complex valued functions whose domains are  the largest sets where they can be defined ($\overline{\mathbb{C}}\setminus \sigma_{{\bf H}}$), with form:
\begin{equation}\label{resolvente}
f_j(z, {\bf H})=\left(z \mathbb{I}-{\bf H}\right)^{-1}{\bf e}_{j-1} \cdot {\bf e}_0, \quad j=1,\ldots,m.
\end{equation}
The following statement in Lemma \ref{resolventbahavior} is for linear operators with forms of ${\bf K}$

\begin{lemma}\label{resolventbahavior} Fix $j\in \{1,\ldots,m\}$. The resolvent function $f_j(\cdot, {\bf K})$ is holomorphic on $\overline{\mathbb{C}}\setminus \sigma_{\bf K}$, where $\overline{\mathbb{C}}=\mathbb{C}\cup \{\infty\}$ is the Riemann sphere. Denote $\displaystyle f_j(\cdot, {\bf K}) \in \mathcal{H}\left(\overline{\mathbb{C}}\setminus \sigma_{\bf K}\right)$. The function $f_j(\cdot, {\bf K})$ also behave close to infinity as
\begin{equation}\label{restrepo}
f_j(z, {\bf K})=\mathcal{O}\left(\frac{1}{z^{j}}\right) \quad \mbox{as} \quad z \to \infty.
\end{equation}
\end{lemma}

\begin{proof} By construction $\displaystyle f_j (\cdot, {\bf K})\in \mathcal{H}\left(\overline{\mathbb{C}}\setminus \sigma_{\bf K}\right)$. According to Lemma \ref{canchero} we have that  $\displaystyle \sigma_{\bf K}\subset \left\{z\in \mathbb{C}: |z|\leq 1+\lambda \mu^m \right\}\cup \left\{z\in \mathbb{C}: |z+\gamma|\leq 1+\lambda \mu^m \right\}$ is bounded. We then apply Neumann's series inversion  (see for instance \cite{Pip}) in the sense of the weak $\ell_2$ topology:
\[
\displaystyle \left(z \mathbb{I}-{\bf K}\right)^{-1}{\bf u}=\left(\left(z+\gamma\right) \mathbb{I}-\left({\bf K}+\gamma \mathbb{I}\right)\right)^{-1}{\bf u}=\sum_{\nu=0}^{\infty} \frac{({\bf K}+\gamma\mathbb{I})^{\nu} {\bf u}}{(z+\gamma)^{\nu+1}}, \quad {\bf u}\in \ell_2,
\]
to obtain the following expression that holds uniformly in any compact set contained in the widest open disk of $\overline{\mathbb{C}}$ centered at $\infty$ leaving outside $\sigma_{\bf K}$
\[
f_j(z, {\bf K})=\sum_{\nu=0}^{\infty} \frac{({\bf K}+\gamma\mathbb{I})^{\nu} {\bf e}_{j-1}\cdot {\bf e}_0}{(z+\gamma)^{\nu+1}}.
\]
From (\ref{gens}) in Lemma \ref{tiberio} we have that for each $\nu\in \{0,\ldots, j-2\}$
\[
 \displaystyle ({\bf K}+\gamma\mathbb{I})^{\nu} {\bf e}_{j-1}\cdot {\bf e}_0=\left(\sum_{k=0}^{\nu}{\nu \choose k} \left(\lambda \mu^m\right)^{\nu-k} {\bf e}_{(\nu-k)m+j-1-k}\right)\cdot {\bf e}_0=0,
\]
hence
\[
f_j(z, {\bf K})=\sum_{\nu=j-1}^{\infty} \frac{({\bf K}+\gamma \mathbb{I})^{\nu} {\bf e}_{j-1}\cdot {\bf e}_0}{(z+\gamma)^{\nu+1}}= \frac{1}{(z+\gamma)^{j-1}}\sum_{\nu=0}^{\infty} \frac{({\bf K}+\gamma\mathbb{I})^{\nu+j-1} {\bf e}_{j-1}\cdot {\bf e}_0}{(z+\gamma)^{\nu+1}},
\]
which proves the statement. 
\end{proof}

Fixed a linear operator as ${\bf H}$ we consider its  sequence of moments $\displaystyle \left\{ c_{\nu,j, {\bf H}}\right\}_{\nu\in \mathbb{Z}_+}$, $j=1,\ldots,m$, as follows
\begin{equation}\label{momentdefinition}
c_{\nu, j, {\bf H}}=\left({\bf H}\right)^{\nu} {\bf e}_{j-1} \cdot {\bf e}_0, \quad \nu \in \mathbb{Z}_+, \quad j=1, \ldots, m.
\end{equation}
In the case of a linear operator as ${\bf K}$ we state the following result

\begin{lemma}\label{firstmoments} Given $j\in \{1, \ldots, m\}$
\[
c_{\nu, j, {\bf K}}=\left({\bf K}\right)^{\nu} {\bf e}_{j-1} \cdot {\bf e}_0=0, \quad \nu=0, \ldots, j-1, \quad j=1, \ldots, m.
\]
\end{lemma}

\begin{proof} Fix $j \in \{1,\ldots,m\}$. Since $\displaystyle \sigma_{\bf K}$ is bounded we apply Neumann's series inversion to $\displaystyle \left(z \mathbb{I}-{\bf K}\right)^{-1}{\bf e}_j$ such that at any compact set contained in the widest open disk of $\overline{\mathbb{C}}$ centered at $\infty$ leaving outside $\sigma_{\bf K}$, the following equality holds uniformly
\[
f_j(z, {\bf K})=\sum_{\nu=0}^{\infty} \frac{{\bf K}^{\nu} {\bf e}_{j-1}\cdot {\bf e}_0}{z^{\nu+1}}=\sum_{\nu=0}^{\infty} \frac{c_{\nu, j, {\bf K}}}{z^{\nu+1}}.
\]
Taking now into account (\ref{restrepo}) we obtain that the first $j-1$th coefficients $c_{\nu, j, {\bf K}}$ must vanish. This completes the proof.
\end{proof}

\subsection{On the sequence $\displaystyle \left\{T_n\right\}_{n\in \mathbb{Z}_+}$}\label{classical}

This Section deals with the sequence of polynomials $\displaystyle \left\{T_n\right\}_{n\in \mathbb{Z}_+}$ satisfying the recurrence relation (\ref{recurrenceT}) with initial conditions (\ref{recurrenceTI}). Some of the results that we present in this Section have  appeared scattered and not detailed in publications such as  \cite{AKS, AKV, AL1, AL2, AE}.  In order to make easier the reading we state these results, accommodating them to our notation, and giving proofs that make comfortable the understanding of some statements in coming Sections.

\begin{lemma}\label{zeroslocation} Set $n\in \mathbb{Z}_+$ with $n=d(m+1)+r$ where $d \in \mathbb{Z}_+$ and $r \in \{0,1,\ldots, m\}$. There exists a polynomial $h_n$ with degree $d$ such that $T_n(z)=z^{r} h_n(z^{m+1})$. The polynomial elements of the sequence $\displaystyle \left\{h_n\right\}_{n\in \mathbb{Z}_+}$ satisfy the following initial conditions and recurrence relations, respectively
\begin{equation}\label{initialh}
h_0=h_1=\cdots=h_m=1
\end{equation}
and when $d \geq 1$
\begin{equation}\label{recurrenceh}
h_{n+1}(z)=\left\{\begin{array}{l c l}
h_n(z)-\mu \lambda^m h_{n-m}(z)  & \mbox{if} & r \in \{0,1,\ldots, m-1\}, \\ & & \\
z \, h_n(z)-\mu \lambda^m h_{n-m}(z) & \mbox{if} & r=m.
\end{array}
\right.
\end{equation}
\end{lemma}

\begin{proof} The initial conditions (\ref{recurrenceTI}) make evident the statement for $n\in \{0,1,\ldots,m\}$. Observe that
\[
T_n(z)=z^n \implies h_n \equiv 1 \quad \mbox{and} \quad r=n,
\] 
which proves (\ref{initialh}). Let us use the induction method. Given $n=d(m+1)+r\geq m$ we assume that the statement holds for every $\widetilde{n}\in \{n-m,n-m+1, \ldots, n\}$. Observe that
\[
n-m=d(m+1)+r-m=d(m+1)-(m+1)+r+1=(d-1)(m+1)+r+1.
\]
We divide the proof in two cases
\begin{enumerate}
\item[]
\item {\bf Case $r=0,1,\ldots,m-1$}. We have that $n+1=d(m+1)+r+1$, with $d\in \mathbb{Z}_+$, $r+1\in \{1,\ldots,m\}$.
\[
T_{n+1}(z)=zT_{n}(z)-\mu \lambda^m T_{n-m}=z z^r h_n(z^{m+1})-\mu \lambda^m z^{r+1} h_{n-m}(z^{m+1}),
\]
then
\[
T_{n+1}(z)=z^{r+1} h_n(z^{m+1})-\mu \lambda^m z^{r+1} h_{n-m}(z^{m+1})
\]
\[
=z^{r+1} \left( h_n(z^{m+1})-\mu \lambda^m h_{n-m}(z^{m+1})\right)=z^{r+1} h_{n+1}(z^{m+1}).
\]
We see that $\deg h_{n+1}=\deg h_n$. We also observe that the equlity
\[
z^{r+1} \left( h_n(z^{m+1})-\mu \lambda^m h_{n-m}(z^{m+1})\right)=z^{r+1} h_{n+1}(z^{m+1}) 
\]
implies that
\[
h_{n+1}(z)=h_n(z)-\mu \lambda^m h_{n-m}(z).
\]
This completes this part of the proof for cases $r=0,1,\ldots,m-1$.
\item[]
\item  {\bf Case $r=m$}. Since $n=d(m+1)+m$ then 
\[
n+1=d(m+1)+m+1=(d+1)(m+1)
\]
and
\[
n-m=d(m+1)+m-m=d(m+1).
\]
Hence
\[
T_{n+1}(z)=zT_{n}(z)-\mu \lambda^m T_{n-m}=z z^m h_n(z^{m+1})-\mu \lambda^m z^{m+1} h_{n-m}(z^{m+1})
\]
\[
=z^{m+1}h_n(z^{m+1})-\mu \lambda^m h_{n-m}(z^{m+1}))=h_{n+1}(z^{m+1}).
\]
This implies that $\deg h_{n+1}=\deg h_n+1=d+1$, and the expression
\[
z^{m+1}\left(h_n(z^{m+1}-\mu \lambda^m h_{n-m}(z^{m+1})\right)=h_{n+1}(z^{m+1})
\]
yields
\[
h_{n+1}(z)=z\, h_n(z)-\mu \lambda^m h_{n-m}(z).
\]
\end{enumerate}
This completes the proof.
\end{proof}

We now study the roots' distribution of the polynomials $h_n$, $n\in \mathbb{Z}_+$.

\begin{lemma}\label{roedores} Consider the sequence $\displaystyle \{h_n\}_{n \in \mathbb{Z}_+}$ as in  Lemma \ref{zeroslocation}. Their elements satisfy the following statements
\begin{enumerate}
\item[]
\item $h_n(0)\not =0$, $n\in \mathbb{Z}_+$.
\item[]
\item Set $n=d(m+1)+r$ with $d\in \mathbb{Z}_+$ and $r \in \{0,1,\ldots, m\}$. Then the $d$ zeros of $h_n$ are real and simple. 
\item[]
\item Fix $d \in \mathbb{Z}_+$, for each $r \in \{0,1,\ldots,m\}$ and denote $x_{n,1} < x_{n,2} < \cdots < x_{n,d}$ the zeros of  $h_n$ with $n=d(m+1)+r$. Set $d\geq 1$ and $n=(d-1)(m+1)+m$ ($r=m$). Then for each $j\in \{1, \ldots, d\}$
\[
x_{n+1,j}<\cdots <x_{n+m+1,j}< x_{n,j}<x_{n+1,j+1}<\cdots <x_{n+m+1,j+1}<x_{n,j+1}.
\]
\item[]
\end{enumerate}
\end{lemma}

\begin{proof} We use induction method. When $n\in \{0,\ldots, m\},$ $d=0$ and from (\ref{initialh}) the first two statements (1) and (2) are  trivially satisfied. Consider $d=1$. From recurrence relations in (\ref{recurrenceh}) we have that
\[
\begin{array} {lcl}
h_{m+1}(z)&=&z\, h_m-\mu \lambda^m \, h_0=z-\mu \lambda^m,\\ & &\\
h_{m+2}(z)&=&h_{m+1}-\mu \lambda^m \, h_1=z-2\mu \lambda^m,\\ & &\\
h_{m+3}(z)&=&h_{m+2}-\mu \lambda^m \, h_2=z-3\mu \lambda^m,\\ & &\\
&\vdots& \\ & &\\
h_{2m+1}(z)&=&h_{2m}-\mu \lambda^m \, h_m=z-(m+1)\mu \lambda^m,
\end{array}
\]
hence
\[
x_{m+k,1}=k\, \mu \lambda^m >0 , \quad k=1,2,\ldots, m+1.
\]
Observe that if $d-1=0$ $h_m\equiv 1$ does not vanish. This implies that the two statements are trivially true
\[
x_{m+1,1}< x_{m+2,1}<\cdots < x_{2m+1,1}.
\]
We assume that the three statements in Lemma \ref{roedores} when $d \geq 1$ are true. Consider $n=d(m+1)+m$. Combining the recurrence relations in (\ref{recurrenceh}) for each $ k=1,2,\ldots, m+1$ we have that
\begin{equation}\label{doniesk}
h_{n+k}(z)=zh_n(z)-\lambda\mu^m \left(h_{n-m}(z)+h_{n-m+1}(z)+\cdots+h_{n-m+k-1}(z)\right).
\end{equation}
From assumption we have that the polynomial 
\[
\mu \lambda^m \left(h_{n-m}(z)+h_{n-m+1}(z)+\cdots+h_{n-m+k-1}(z)\right)
\]
does not vanish at the origin, hence the equation  (\ref{doniesk})  yields
\[
h_{n+k}(0)=-\mu \lambda^m \left(h_{n-m}+h_{n-m+1}+\cdots+h_{n-m+k-1}\right)(0)\not =0.
\]
The property (1) is already proved. We continuous with the two others (2) and (3).

Note that
\[
h_{n+k}(x_{n,j})=-\mu \lambda^m \left(h_{n-m}+h_{n-m+1}+\cdots+h_{n-m+k-1}\right)(x_{n,j}), \quad j=1,2,\ldots,d.
\]
Then the polynomial $h_{n+k}$ also changes sign at a point between two consecutive zeros of $h_n$. Observe that all the polynomials $h_n$ are monic. Let us take into account that
\[
\deg h_{n+k}=\deg \mu \lambda^m \left(h_{n-m}+h_{n-m+1}+\cdots+h_{n-m+k-1}\right)+2.
\]
This implies that
\[
\lim_{z\to -\infty} h_{n+k}(z)=\lim_{z\to -\infty} \lambda\mu^m \left(h_{n-m}+h_{n-m+1}+\cdots+h_{n-m+k-1}\right)(z)=\pm \infty.
\]
Assuming that the polynomial $\mu \lambda^m \left(h_{n-m}+h_{n-m+1}+\cdots+h_{n-m+k-1}\right)(z)$ does not vanish before $x_{n,1}$, and taking into account that the polynomials  $h_{n+k}(x_{n,1})$ and $\mu \lambda^m \left(h_{n-m}+h_{n-m+1}+\cdots+h_{n-m+k-1}\right)(x_{n,1})$
\[
h_{n+k}(x_{n,1})=-\mu \lambda^m \left(h_{n-m}+h_{n-m+1}+\cdots+h_{n-m+k-1}\right)(x_{n,1}),
\]
have different sing, hence $h_{n,k}$ changes sing ones before $x_{n,1}$. We have proved that given $k\in \{2,\ldots, m+1\}$, $x_{n,k-1}<x_{n+1,k}<x_{n,k}$ and $x_{n+1,1}<x_{n,1}$. 

In order to complete the proof we only need to show that given $k\in \{1,\ldots, m\}$, $x_{n+k,j}<x_{n+k+1,j}$.  We procede similarly as we have done. Consider again the recurrence relations in (\ref{recurrenceh}). We have that
\[
h_{n+k+1}=h_{n+k}(z)-\mu \lambda^m h_{n-m+k}(z).
\]
Let us recall the assumption that the polynomial $h_{n-m}$ changes sign between two consecutive zeros of $h_{n+k}$, then  $h_{n+k+1}$ also interlace zeros with $h_{n+k}$. Since $h_{n-m+k}$ does not vanishes after $x_{n+k,d}$ and taking into account that both $h_{n+k}$ and $h_{n-m+k}$ are monic polynomials and also from the fact that they have different sign at  $x_{n+k,d}$, we have that $h_{n+k+1}$ vanishes ones after $x_{n+k,d}$. This completes the proof.
\end{proof}

Combining Lemma \ref{canchero}, Lemma \ref{zeroslocation}, and Lemma \ref{roedores} we have that given any $n=d(m+1)+r$ its corresponding polynomial $T_n(z)=z^{r}h_n(z^{m+1})$ has a zero of order $r$ at the origin and $d(m+1)$ simple roots more shared in a starlike set $\Sigma$ with
\[
\displaystyle \Sigma \subset \bigcup_{k=0}^{m} \left[ -\left(1 +\mu \lambda^n\right)^{1/m+1} \exp \left(\frac{2 \pi i k}{m+1}\right) \,, \, \left(1+ \mu\lambda^n\right)^{1/m+1} \exp \left(\frac{2 \pi i k}{m+1} \right)\right].
\]
There are $d$ different real zeros 
\[
x_{n,j}\in  \left[ -1- \mu \lambda^m\,, \, 1+ \mu \lambda^m\right], \quad j=1,2,\ldots,d,
\]
and for each $k \in \{1,\ldots, m\}$, $T_n$ vanishes at
\[
x_{n,j} \exp \left(\frac{2 \pi i k}{m+1} \right), \quad j=1,\ldots, d.
\]
Studying the strong asymptomatic behavior of the polynomials $T_n$ we prove that the set  
\[
S_0=\bigcup_{k=0}^m \left[0,\, a\exp{\frac{2 \pi i k}{m+1}}\right], \quad a=\frac{m+1}{m}\left(\frac{\mu \lambda^m}{m}\right)^{m/m+1}.
\]
attracts the points $\displaystyle x_{n,j}$, $j=1,2,\ldots,d$, $n \in \mathbb{Z}_+$. 

In order to analyze the strong asymptotic behavior of the sequence $\displaystyle \left\{T_n\right\}_{n\in \mathbb{Z}_+}$ we need to give some properties of the certain solutions $\omega_0, \ldots, \omega_m$ corresponding to the algebraic equation  
\begin{equation}\label{omegaequation}
\omega^{m+1}-z\omega^m+\mu\lambda^m=0.
\end{equation} 
This equation is well studied in  \cite{AKV}. Since the roots of $T_n$ belong to a starlike set $\Sigma$, we choose the branch for  $\omega_0, \ldots, \omega_m$ whose cuts are as small as possible starlike sets.  Let  $[0,\infty) \times \exp \left(i\theta\right),$ $\theta \in \mathbb{R}$ denote the ray that starts at the origin with slope $\tan \theta$. With the expression  $[-\infty,\infty) \times \exp \left(i\theta\right)$ we refer the straight line with slope  $\tan \theta$ containing the origin. Let us declare the following starlike sets
\begin{equation}\label{Sset}
\begin{array}{l}
\displaystyle S_0=\bigcup_{k=0}^m \left[0,\, a\exp{\frac{2 \pi i k}{m+1}}\right], \\ \\
 \displaystyle S_{\infty}^e=\bigcup_{k=0}^m  [-\infty,\,\infty) \times \exp \left(\frac{(2k+1) i \pi}{m+1}\right),\\ \\
\displaystyle S_{\infty}^o=\bigcup_{k=0}^m  [0,\infty) \times \exp \left(\frac{2k i \pi}{m+1}\right).
\end{array}
\end{equation}
We denote
\[
\Omega_0=\mathbb{C}\setminus S_0, \quad \Omega_j=\mathbb{C}\setminus \{S_{\infty}^o\cup S_{\infty}^e\}, \quad  1 \leq j< m \quad \mbox{and} \quad \Omega_m=\left\{\begin{array}{l l}
\mathbb{C} \setminus S_{\infty}^e &  \mbox{if  } m \mbox{ even},\\
 \mathbb{C} \setminus S_{\infty}^o &  \mbox{if } m \mbox{ odd},
\end{array}\right.
\]
where
\[
a=\frac{m+1}{m} \left(\frac{\mu\lambda^m}{m}\right)^{m/m+1}.
\]
We then choose  for each $\displaystyle j\in \left\{0,1,\ldots,m\right\}$, $\omega_j$ is a holomorphic function on $\Omega_j$ (Notation:  $\omega_j \in \mathcal{H}(\Omega_j)$) and
\begin{equation}\label{maltrecho}
\omega_0(z)=z+\mathcal{O}(1) \quad \mbox{and} \quad \omega_j(z)=\mathcal{O}\left(\frac{1}{z^{1/m}}\right),  \quad \mbox{as} \quad z\to \infty.
\end{equation}
From \cite[Proposition 1]{AKV}\label{solutionproperties} we have that 
\begin{equation}\label{inequalitycomplex}
|\omega_0|\geq |\omega_1|\geq \cdots \geq |\omega_m|, \quad \mbox{on }\, \mathbb{C}.
\end{equation}
In $\Omega_0$ the inequalities are strict.

We are now ready to find the strong asymptotic behavior of the family of polynomials $T_n$, $n\in \mathbb{Z}_+$.

\begin{lemma}\label{Ithaca} Let $\displaystyle \left\{T_n\right\}_{n\in \mathbb{Z}_+}$ be the sequence of polynomials defined by the recurrence relation (\ref{recurrenceT}) with initial conditions (\ref{recurrenceTI}). The following equality hold uniformly for any compact set $K\subset \Omega_0$
\begin{equation}\label{asymptoticT}
\lim_{n\to \infty} \frac{T_n(z)}{\omega_0(z)}=\frac{1}{1-\frac{m \lambda^m \mu}{\omega_0^{m+1}(z)}}.
\end{equation}
Consequently for any open set $D$ that contains $S_0$ there is a number $N\in \mathbb{Z}_+$ such that for all $n \geq N$ the polynomial $T_n$ has its roots in $D$. 
\end{lemma}

\begin{remark}  Since the boundary of $D$ could be as close to $S_0$ as one wants, this statement means that $S_0$ attracts the toots of the polynomials $T_n$, $n\in \mathbb{Z}_+$. Using this result we shall prove that $S_0=\sigma_{\bf T}$.
\end{remark}

\begin{proof} Let us recall the equations (\ref{recurrenceTN})  and (\ref{recurrenceTIN}) to define the polynomials $T_n,$ $n \in \mathbb{Z}_+$. We write $T_n$ as follows
\[
T_n=a_0\omega_0^{n+m}+a_1\omega_1^{n+m}+\cdots+a_m \omega_m^{n+m}, \quad n \in \mathbb{Z}_+
\]
whose coefficients $a_j$, $j\in \{0,1,\ldots, m\}$ are real valued functions that satisfy the initial conditions
\[
a_0\omega_0^{n}+a_1\omega_1^{n}+\cdots+a_m \omega_m^{n}=\delta_{n,m}, \quad n \in \{0,1,2,\ldots,m\}.
\]
This is a linear system of equation whose unknown are exactly the coefficients $a_j$, $j\in \left\{0,1,\ldots,m\right\}$. It's corresponding matrix equation is ${\bf M}{\bf x}={\bf b}$ where
\[
{\bf M}=\left(\begin{array}{c c c c}
1 & 1 & \cdots & 1\\ & & & \\
\omega_0 & \omega_1 & \cdots & \omega_m \\ & & & \\
\vdots & \vdots & \ddots & \vdots \\ & & & \\
\omega_0^{m} &\omega_1^m & \cdots & \omega_m^m 
\end{array}
\right), \quad {\bf x}=\left(\begin{array}{c}
a_0 \\  \\
a_1 \\ \\
\vdots \\  \\
a_m
\end{array}
\right), \quad \mbox{and} \quad {\bf b}=\left(\begin{array}{c}
0 \\  \\
\vdots \\ \\
0 \\  \\
1
\end{array}
\right).
\] 
Using Cramer's rule we obtain that for each $\displaystyle j\in \left\{0, 1,\ldots, m\right\}$
\[
a_j=\frac{(-1)^{j}}{\displaystyle \prod_{k\not= j}(\omega_j-\omega_k)}.
\]
Observe that
\[
P(\omega)= \prod_{k=1}^m(\omega-\omega_k)=\omega^{m+1}-z\,\omega^m+\lambda^m \mu \quad P(\omega_j)=0.
\]
We analyze its derivative
\[
P^{\prime}(\omega)=(m+1)\omega^m-zm\omega^{m-1}=\frac{m}{\omega}\left(\omega^{m+1}-z\, \omega^m\right)+\omega^m,
\]
then
\[
P^{\prime}(\omega_j)=\omega_j^m-\frac{m\lambda^m\mu}{\omega_j}=\frac{\omega_j^{m+1}-m\lambda^m\mu}{\omega_j}=\prod_{k\not= j}(\omega_j-\omega_k).
\]
Hence
\[
a_j=\frac{(-1)^j\omega_j}{\omega_j^{m+1}-m\lambda^m\mu}=\left(\frac{(-1)^j}{1-\frac{m \lambda^m\mu}{\omega_j^{m+1}}}\right) \frac{1}{\omega_j^{m}}, \qquad j\in \{0,1, \ldots, m\}.
\]
We have that
\begin{equation}\label{formulaT}
T_n(z)=\sum_{j=0}^m \frac{(-1)^j \omega_j^n(z)}{1-\frac{m \lambda^m \mu}{\omega_j^{m+1}(z)}}.
\end{equation}

From the inequalities (\ref{inequalitycomplex}) we obtain that given any compact $K \subset \Omega_0$ 
\[
\lim_{n\to \infty} \frac{T_n(z)}{\omega_0(z)}=\frac{1}{1-\frac{m \lambda^m \mu}{\omega_0^{m+1}(z)}},
\]
which completes the proof.
\end{proof}

\subsection{Integral expressions of  $f_j(\cdot, {\bf T})$}

We obtain integral expressions for $f_j(\cdot, {\bf T})$, $j=1,\ldots, m$. We star by introducing the so called second kind polynomials

\begin{equation}\label{secondkindpolLT}
T_{n,j}(z)=\left(z \mathbb{I}-{\bf T}\right)^{-1}{\bf e}_{j-1} \cdot \left(T_n(z \mathbb{I})  - T_n({\bf T}^{\top})\right)\, {\bf e}_0 , \quad j=1,\ldots,m.
\end{equation}

\begin{lemma}\label{visual} Let $n \in \mathbb{Z}_+$ and $j\in \{1,\ldots, m\}$ be given, with $n=\kappa\, m+s$, $(\kappa,s)\in \mathbb{Z}_+\times\{0,1,\ldots, m-1\}$. The resolvent functions in (\ref{resolvente}) corresponding to the linear operator ${\bf T}$ and the polynomial (\ref{secondkindpolLT}) satisfy the identity
\begin{equation}\label{identityKf}
T_n(z)f_j(z, {\bf T})-T_{n,j}(z)=\left(z\mathbb{I}-{\bf T}\right)^{-1} {\bf e}_{j-1}\cdot T_n({\bf T}^{\top}) \,{\bf e}_0,
\end{equation}
where
\begin{equation}\label{analiticidadremainder}
\left(z\mathbb{I}-{\bf T}\right)^{-1} {\bf e}_{j-1}\cdot T_n({\bf T}^{\top}) \,{\bf e}_0\in \mathcal{H}\left(\overline{\mathbb{C}}\setminus \sigma_{\bf T}\right),
\end{equation}
and 
\begin{equation}\label{remainderbehavior}
\left(z\mathbb{I}-{\bf T}\right)^{-1} {\bf e}_{j-1}\cdot T_n({\bf T}^{\top}) \,{\bf e}_0=\mathcal{O}\left(\frac{1}{z^{\tau(n,j)}}\right) \quad \mbox{as}\quad z\to \infty
\end{equation}
with
\[
\tau(n,j)=\left\{\begin{array}{l l l}
\kappa+m+j-s+1 & \mbox{if} & j\in \{1,\ldots, s\}, \\ & & \\
\kappa+j-s+1 & \mbox{if} & j\in \{s+1,\ldots,m\}.
\end{array}
\right.
\]
\end{lemma}

\begin{proof} Fix $j\in \{1,2,\ldots,m\}$. Expression (\ref{secondkindpolLT}) yields immediately identity (\ref{identityKf}). From Lemma \ref{resolventbahavior} we have that the resolvent function $\displaystyle f_j(\cdot, {\bf T})\in \mathcal{H}\left(\overline{\mathbb{C}}\setminus \sigma_{\bf T}\right)$. So taking into account (\ref{identityKf}) we show (\ref{analiticidadremainder}). 

Since $\sigma_{\bf T}$ is bounded we apply again Neumann's series for operator inversion to  $\displaystyle \left(z \mathbb{I}-{\bf T}\right)^{-1}$, in the sense of weak $\ell_2$ topology, and we obtain the following equality uniformly on any compact set of the maximum open disk in $\overline{\mathbb{C}}$ centered at $\infty$ which leaves outside $\mbox{supp}\left({\bf T}\right)$. Hence
\[
\left(z\mathbb{I}-{\bf T}\right)^{-1} {\bf e}_{j-1}\cdot T_n({\bf T}^{\top})\, {\bf e}_0=\sum_{\nu=0}^{\infty} \frac{ {\bf T}^{\nu} {\bf e}_{j-1}\cdot T_n({\bf T}^{\top})\, {\bf e}_0}{z^{\nu+1}}, \quad j=1,\ldots, m.
\]
Using the identity  in (\ref{comite}) for $T_n$, we obtain that
\[
\left(z\mathbb{I}-{\bf T}\right)^{-1} {\bf e}_{j-1}\cdot T_n({\bf T}^{\top})\, {\bf e}_0=\sum_{\nu=0}^{\infty} \frac{ {\bf T}^{\nu} {\bf e}_{j-1}\cdot {\bf e}_n}{z^{\nu+1}}.
\]
Recall (\ref{gens}) in Lemma \ref{tiberio} 
\[
\left(z\mathbb{I}-{\bf T}\right)^{-1} {\bf e}_{j-1}\cdot T_n({\bf T}^{\top})\, {\bf e}_0=\sum_{\nu=0}^{\infty} \frac{1}{z^{\nu+1}} \sum_{k=0}^{\nu}{\nu \choose k} \left(\mu \lambda^m\right)^{\nu-k} {\bf e}_{(\nu-k)m+j-1-k} \cdot {\bf e}_n
\]
\[
=\sum_{\nu=0}^{\infty} \frac{1}{z^{\nu+1}} \sum_{k=0}^{\nu}{\nu \choose k} \left(\mu \lambda^m\right)^{\nu-k} {\bf e}_{(\nu-k)m+j-1-k} \cdot {\bf e}_{\kappa m+s}
\]
\[
=\sum_{\nu=0}^{\infty} \frac{1}{z^{\nu+1}} \sum_{k=0}^{\nu}{\nu \choose k} \left(\mu \lambda^m\right)^{\nu-k} \delta_{(\nu-k)m+j-1-k, \kappa m+s}.
\]
Let us start by analyzing the cases $j\in \{1,\ldots,s\}$. In above summation the first term which does not vanish occurs when $\nu-k-1=s$ and $m+j-1-k=s$, hence $\nu=\kappa+m+j-s$. This implies that
\[
\left(z\mathbb{I}-{\bf T}\right)^{-1} {\bf e}_{j-1}\cdot T_n({\bf T}^{\top})\, {\bf e}_0= \mathcal{O}\left(\frac{1}{z^{\kappa+m+j-s+1}}\right) \quad \mbox{as} \quad z \to \infty.
\]
When $j\in \{s+1,\ldots, m\}$, the first term that is different from zero, corresponds to $\nu-k-1=\kappa$ and $j-1-k$, so $\nu=\kappa+j-s$ and 
\[
\left(z\mathbb{I}-{\bf T}\right)^{-1} {\bf e}_{j-1}\cdot T_n({\bf T}^{\top})\, {\bf e}_0=\sum_{\nu=0}^{\infty} \frac{{\bf T}^{\nu} {\bf e}_{j-1}\cdot {\bf e}_n}{z^{\nu+1}}= \mathcal{O}\left(\frac{1}{z^{\kappa+j-s+1}}\right), \,\,\mbox{as}\, z\to \infty.
\]
These equalities complete the proof.
\end{proof}

\begin{lemma}\label{LemmarecurrenceP} Fix $\displaystyle j\in \left\{1, \ldots, m\right\}$. The sequence of polynomials $\displaystyle \left\{T_{n,j}\right\}_{n\in \mathbb{Z}_+}$ in (\ref{secondkindpolLT}) satisfy the following recurrence relations
\begin{equation}\label{biquini}
zT_{n,j}(z)=T_{n+1,j}(z)+\mu \lambda^m T_{n-m,j}(z),
\end{equation}
with the initial conditions
\begin{equation}\label{initial}
T_{n,j} \equiv 0, \quad n=0, 1, \ldots, j-1 \quad \mbox{and} \quad T_{n,j}(z)=z^{n-j}, \quad n=j, \ldots,m.
\end{equation}
\end{lemma}

\begin{proof} For each $\displaystyle j\in \{1,\ldots,m\}$, consider  (\ref{identityKf})
\[
T_{n,j}(z)=T_n(z)f_j(z, {\bf T})-\left(z\mathbb{I}-{\bf T}\right)^{-1} {\bf e}_{j-1}\cdot T_n({\bf T}^{\top}) \,{\bf e}_0.
\]
Let us recall Lemma \ref{resolventbahavior}. We have that $\displaystyle f_j(\cdot, {\bf T})\in \mathcal{H}\left(\overline{\mathbb{C}}\setminus \sigma_{\bf T}\right)$. Combining this fact with (\ref{analiticidadremainder}) and (\ref{remainderbehavior})  in Lemma \ref{visual}, we conclude that for any $n\in \{0,\ldots, j-1\}$ the polynomial
\[
T_{n,j}(z)=T_n(z)f_j(z, {\bf T})-\left(z\mathbb{I}-{\bf T}\right)^{-1} {\bf e}_{j-1}\cdot T_n({\bf T}^{\top}) \,{\bf e}_0
\]
\[
=\mathcal{O}\left(\frac{1}{z}\right)+\mathcal{O}\left(\frac{1}{z^{\tau(n,j)}}\right)=\mathcal{O}\left(\frac{1}{z}\right) \quad \mbox{as}\quad z\to \infty,
\]
which yields $\displaystyle T_{n,j}\equiv 0,$ $n=0, \ldots, j-1$. When $n=j$ we have that the polynomial $T_{n,j}$ is bounded in $\overline{\mathbb{C}}$, hence its a constant. From (\ref{gens}) in Lemma \ref{tiberio} we obtain that 
\[
\lim_{z\to \infty }T_{n,j}(z)=\lim_{z\to \infty} T_n(z)f_j(z)=\lim_{z\to \infty} z^j \frac{1}{z^j} {\bf e}_{j-1} \cdot \left({\bf T}^{\top}\right)^{ j} {\bf e}_0=1.
\]
We now consider $n \in \{j, \ldots, m-1\}$ and with  (\ref{recurrenceKI}) 
\[
z T_{n,j}(z)=z\left(z \mathbb{I}-{\bf T}\right)^{-1}{\bf e}_{j-1} \cdot \left(T_n(z \mathbb{I})  - T_n({\bf T}^{\top})\right)\, {\bf e}_0
\]
\[
=\left(z \mathbb{I}-{\bf T}\right)^{-1}{\bf e}_{j-1} \cdot \left(z T_n(z \mathbb{I})  - z T_n({\bf T}^{\top})\right)\, {\bf e}_0
\]
\[
=\left(z \mathbb{I}-{\bf T}\right)^{-1}{\bf e}_{j-1} \cdot \left(z T_n(z \mathbb{I})  - (z\mathbb{I}-{\bf T}^{\top}) \, T_n({\bf T}^{\top})-({\bf T}^{\top}+\iota\mathbb{I})T_n({\bf T}^{\top})\right)\, {\bf e}_0
\]
\[
=\left(z \mathbb{I}-{\bf T}\right)^{-1}{\bf e}_{j-1} \cdot \left(z\mathbb{I} T_n(z \mathbb{I})  - (z\mathbb{I}-{\bf T}^{\top}) \, T_n({\bf T}^{\top})-{\bf T}^{\top}\,T_n({\bf T}^{\top})\right)\, {\bf e}_0
\]
\[
=\left(z \mathbb{I}-{\bf T}\right)^{-1}{\bf e}_{j-1} \cdot \left(T_{n+1}(z \mathbb{I}) -T_{n+1}({\bf T}^{\top})\right)\, {\bf e}_0
\]
\[
-\left(z \mathbb{I}-{\bf T}\right)^{-1}{\bf e}_{j-1} \cdot \left((z\mathbb{I}-{\bf T}^{\top}) \, T_n({\bf T}^{\top})\right)\, {\bf e}_0
\]
\[
=T_{n+1,j}(z)-\left(z \mathbb{I}-{\bf T}\right)^{-1}{\bf e}_{j-1} \cdot \left((z\mathbb{I}-{\bf T}^{\top}) \, T_n({\bf T}^{\top})\right)\, {\bf e}_0
\]
\[
=T_{n+1,j}(z)-{\bf e}_{j-1}^{\top} \,\left(z \mathbb{I}-{\bf T}^{\top}\right)^{-1} (z\mathbb{I}-{\bf T}^{\top}) \, T_n({\bf T}^{\top})\, {\bf e}_0=T_{j, n+1}(z)-{\bf e}_{j-1}^{\top} \, {\bf e}_n.
\]
Since $n\geq j$ we obtain (\ref{initial}).

Let us prove (\ref{biquini}). Set $n\geq m$ and recall that (\ref{recurrenceK})
\[
z\, T_{n,j}(z)=z \left(z \mathbb{I}-{\bf T}\right)^{-1}{\bf e}_{j-1} \cdot \left(T_n(z \mathbb{I})  - T_n({\bf T}^{\top})\right)\, {\bf e}_0
\]
\[
=\left(z \mathbb{I}-{\bf T}\right)^{-1}{\bf e}_{j-1} \cdot \left(z\mathbb{I} T_n(z \mathbb{I})  - z\mathbb{I} T_n({\bf T}^{\top})\right)\, {\bf e}_0
\]
\[
=\left(z \mathbb{I}-{\bf T}\right)^{-1}{\bf e}_{j-1} \cdot \left(z\mathbb{I} T_n(z \mathbb{I})  - (z\mathbb{I}-{\bf T}^{\top}) T_n({\bf T}^{\top})-{\bf T}^{\top} T_n({\bf T}^{\top})\right)\, {\bf e}_0
\]
\[
=\left(z \mathbb{I}-{\bf T}\right)^{-1}{\bf e}_{j-1} \cdot \left(T_{n+1}(z \mathbb{I})+\mu \lambda^m\,T_{n-m}(z \mathbb{I}) -T_{n+1}({\bf T}^{\top})-\mu \lambda^m T_{n-m}({\bf T})\right)\, {\bf e}_0
\]
\[
-\left(z \mathbb{I}-{\bf T}\right)^{-1}{\bf e}_{j-1} \cdot \left((z\mathbb{I}-{\bf T}^{\top}) \, T_n({\bf T}^{\top})\right)\, {\bf e}_0=T_{ n+1,j}(z)+\mu \lambda^m T_{n-m,j}(z). 
\]
This proves (\ref{biquini}). 
\end{proof}

Let us consider the second kind polynomials corresponding to $T_n$ defined in (\ref{secondkindpolLT}) denote by $T_{n,j}$ $j=1,\ldots,m$. We include in the notation $T_n=T_{0,n}$.

\begin{lemma}\label{motorina} Given $n \in \mathbb{Z}_+$
\begin{equation}\label{shift}
T_n \equiv T_{n,0}\equiv T_{n+1,1}\equiv \cdots \equiv T_{n+m,m}.
\end{equation}
\end{lemma} 

\begin{proof} Fix $j\in \{0,1,\ldots,m\}$. The recurrence relations in Lemma  \ref{LemmarecurrenceP} can be written as
\begin{equation}\label{initialT}
T_{n,j} \equiv 0, \quad n=j-m, j-m+1, \ldots, j-1 \quad \mbox{and} \quad T_{j,i}(z)=1,
\end{equation}
and
\begin{equation}\label{biquiniT}
zT_{n,j}(z)=T_{n+1,j}(z)+\mu \lambda^m T_{n-m,j}(z), \quad n\geq j.
\end{equation}
The initial conditions in (\ref{initialT}) coincide with the ones in (\ref{initial}) for every $n \in \{0,1,\ldots,j\}$. We now consider the cases $n \in \{j,j+1,\ldots,m-1\}$. Taking into account that $T_{n-m,j}\equiv 0$ we see that the recurrence relation (\ref{biquiniT}) becomes
\[
T_{n,j}(z)=zT_{n-1,j}(z)=z\, z^{n-1}=z^n.
\]
Since the recurrence relations (\ref{biquini}) and (\ref{biquiniT}) coincide we have that for each $j\in \{0,\ldots,m\}$ the sequence $\displaystyle \{T_{n,j}:n \geq j-m\}$ is well defined. This proves (\ref{shift}).
 \end{proof}

We now study an interpolatory quadrature formula in this context of operators

\begin{lemma}\label{abeycanalla} Fix $n=d(m+1)+r$ with $d \in \mathbb{Z}_+$ and $r \in \{0,1,\ldots,m\}$. Set a polynomial $P$ with $\displaystyle \deg P \leq 2(m+1) d+r-1$. Denote
\[
\left(x_{n,j}\right)^{1/m+1}=\zeta_{n,j} \exp \frac{2 \pi i k}{m+1}, \quad k=0,1,\ldots, m, \quad j=1,2, \ldots,d, 
\]
the roots of the composite polynomial $h_n(z^{m+1})$.
\begin{equation}\label{romel}
P({\bf T}) {\bf T}^r {\bf e}_{0} \cdot {\bf e}_0=\sum_{k=0}^m \sum_{j=1}^d \lambda_{n,j,k} \frac{P\left(\zeta_{n,j}\exp \frac{2 \pi i k}{m+1}\right)}{(m+1)\left(\zeta_{n,j}\exp \frac{2 \pi i k}{m+1}\right)^{m+1-r}},
\end{equation}
where $k=0,1,\ldots,m,$ $j=1,\ldots,d,$
\begin{equation}\label{positivecoeficientes}
\lambda_{n,j,k}=\frac{h_{n-1}\left(x_{n,j}\right)}{\displaystyle h_n^{\prime}\left(x_{n,j}\right)} > 0
\end{equation}
such that
\begin{equation}\label{olomodor}
\sum_{k=0}^m \sum_{j=1}^d \lambda_{n,j,k} =(m+1)\mu \lambda^{m}. 
\end{equation}
\end{lemma}

\begin{proof} The Lagrange interpolatory polynomial that interpolates $P$ at the roots of $h_n(z^{m+1})$ has the following form
\[
\mathcal{L}_n\left[P\right]({\bf T})=\sum_{k=0}^m \sum_{j=1}^d \frac{P \left(\zeta_{n,j}\exp \frac{2 \pi i k}{m+1}\right)}{\displaystyle (m+1) \left(\zeta_{n,j}\exp \frac{2 \pi i k}{m+1}\right)^{m}h_n^{\prime}\left(x_{n,j}\right)}
\]
\[
\times  \left({\bf T}-\zeta_{n,j}\exp \frac{2 \pi i k}{m+1}\mathbb{I}\right)^{-1} h_n\left(\left({\bf T}\right)^{m+1}\right).
\]
Form Lemma \ref{roedores} we have the denominator never becomes zero. Recall that when $n=d(m+1)+r$, $\displaystyle T_n(z)=z^r\, h_n\left(z^{m+1}\right)$. Then we obtain that
\begin{equation}\label{hotelcalifornia}
\mathcal{L}_n\left[P\right]({\bf T})\, {\bf T}^r=\sum_{k=0}^m \sum_{j=1}^d \frac{P \left(\zeta_{n,j}\exp \frac{2 \pi i k}{m+1}\right)}{\displaystyle (m+1) \left(\zeta_{n,j}\exp \frac{2 \pi i k}{m+1}\right)^{m}h_n^{\prime}\left(x_{n,j}\right)} 
\end{equation}
\[
\times  \left({\bf T}-\zeta_{n,j}\exp \frac{2 \pi i k}{m+1}\mathbb{I}\right)^{-1} T_n\left({\bf T}\right),
\]
and also there is a polynomial $\mathcal{Q}$ with degree $\deg \mathcal{Q} \leq n-1$
\[
\left[\mathcal{L}_n\left[P\right]({\bf T})-P({\bf T})\right] {\bf T}^r= \mathcal{Q}({\bf T})\, h_n\left({\bf T}^{m+1}\right)\,{\bf T}^r= \mathcal{Q}({\bf T})\, T_n({\bf T}).
\]
Now using the identities (\ref{comite}) in Lemma \ref{jumplemmaQLT} and (\ref{operatorT})
\[
\left[\mathcal{L}_n\left[P\right]({\bf T})-P({\bf T})\right] {\bf T}^r{\bf e}_0 \cdot {\bf e}_0= \mathcal{Q}({\bf T})\, T_n({\bf T}) {\bf e}_0 \cdot {\bf e}_0
\]
\[
=\mathcal{Q}({\bf T}) {\bf e}_0 \cdot \, T_n({\bf T}^{\top}){\bf e}_0=\mathcal{Q}({\bf T}) {\bf e}_0 \cdot \,{\bf e}_n=0
\]
In the last equality we took into account that $m \deg \mathcal{Q}\leq  n-1$ and (\ref{operatorT}). Let us study the equality
\[
P({\bf T}) \left({\bf T}\right)^r{\bf e}_0 \cdot {\bf e}_0=\mathcal{L}_n\left[P\right]({\bf T}) {\bf T}^r{\bf e}_0 \cdot {\bf e}_0.
\]
Form (\ref{hotelcalifornia}) we have that
\[
P({\bf T}) {\bf T}^r{\bf e}_0 \cdot {\bf e}_0=\sum_{k=0}^m \sum_{j=1}^d \frac{P \left(\zeta_{n,j}\exp \frac{2 \pi i k}{m+1}\right)}{\displaystyle (m+1) \left(\zeta_{n,j}\exp \frac{2 \pi i k}{m+1}\right)^{m}h_n^{\prime}\left(x_{n,j}\right)} 
\]
\[
\times\left({\bf T}-x_{n,j}\exp \frac{2 \pi i k}{m+1}\mathbb{I}\right)^{-1} T_n\left({\bf T}\right) {\bf e}_0 \cdot {\bf e}_0,
\]
which combined with (\ref{secondkindpolLT}) yields
\[
P({\bf T}) \left({\bf T}\right)^r{\bf e}_0 \cdot {\bf e}_0
\]
\[
=\sum_{k=0}^m \sum_{j=1}^d \frac{P \left(\zeta_{n,j}\exp \frac{2 \pi i k}{m+1}\right)}{\displaystyle (m+1) \left(\zeta_{n,j}\exp \frac{2 \pi i k}{m+1}\right)^{m}h_n^{\prime}\left(x_{n,j}\right)}  T_{n,1}\left(\zeta_{n,j}\exp \frac{2 \pi i k}{m+1}\right)
\]
\[
=\sum_{k=0}^m \sum_{j=1}^d \frac{\left(\zeta_{n,j}\exp \frac{2 \pi i k}{m+1}\right)^{r-1}P \left(\zeta_{n,j}\exp \frac{2 \pi i k}{m+1}\right)}{\displaystyle (m+1) \left(\zeta_{n,j}\exp \frac{2 \pi i k}{m+1}\right)^{m} h_n^{\prime}\left(x_{n,j}\right)}  h_{n-1}\left(x_{n,j}\right).
\]
which proves (\ref{romel}) and (\ref{positivecoeficientes}). Taking $\displaystyle P\left({\bf T}\right)=\left({\bf T}\right)^{m+1-r}$ we obtain that
\[
\sum_{k=0}^m \sum_{j=1}^d \frac{h_{n-1}\left(x_{n,j}\right)}{\displaystyle (m+1)  h_n^{\prime}\left(x_{n,j}\right)}  ={\bf T}^{m+1-r} {\bf T}^r{\bf e}_0 \cdot {\bf e}_0={\bf e}_0 \cdot  \left({\bf T}^{\top}\right)^{m+1} {\bf e}_0=\mu \lambda^m.
\]
This proves (\ref{olomodor}). The relations in (\ref{positivecoeficientes}) are deduced from Lemma \ref{roedores}. Since $\displaystyle h_{n}$ has simple zeros its derivatives changes sign between two consecutive roots of such polynomial. In addition we observe that $h_{n-1}$ and $h_{n}^{\prime}$ have same sign at $x_{n,1}$, hence they also have same sign at every $x_{n,j}$, $j=1,\ldots, d$, and their quotient is always positive.
\end{proof}

\begin{proposition}\label{propcociente} Let $\displaystyle \left\{T_n\right\}_{n\in \mathbb{Z}_+}$  be the sequence of polynomials satisfying the recurrence relation (\ref{recurrenceT}) with initial conditions (\ref{recurrenceTI}). Given an arbitrary compact set $\mathcal{K}\subset \mathbb{C}\setminus S_0$ the following equality holds uniformly
\begin{equation}\label{quotienelimit}
\lim_{n\to \infty} \frac{T_{n-1}}{T_{n}}=\frac{1}{\omega_0}
\end{equation}
where $\omega_0$ is the branch solution of the algebraic equation  (\ref{omegaequation}) as in (\ref{maltrecho}). 
\end{proposition}

\begin{proof}  From Lemma \ref{Ithaca} we have that the zeros of $T_n$, are attracted by the starlike set  $S_0$. Fix a compact set $\mathcal{K}\subset \Omega_0=\mathbb{C}\setminus S_0$, then there is an $N$ large enough such that any  $N \leq n=d(m+1)+r\in \mathbb{Z}_+$ such that the quotient 
\begin{equation}\label{quotienlimit}
 \frac{T_{n-1}(z)}{T_{n}(z)} =z^{\kappa}\frac{h_{n-1}(z^{m+1})}{h_{n}(z^{m+1})}\in \mathcal{H}\left(\mathcal{K} \right), \quad \kappa \in \{-1,m\}.
\end{equation}
Given $N\in \mathbb{Z}_+$, set $\Sigma_N$ the smallest starlike set that contains all the zeros of the polynomials $T_n$, $n \geq N$. Let us introduce the following magnitudes:
\[
\mbox{dist}(\mathcal{K},\Sigma_N)=\min_{(z,\zeta)\in \mathcal{K}\times\Sigma_N} \left|z-\zeta\right| 
\]
and
\[
M(\mathcal{K}, N)=\left\{\begin{array}{l c l}
\displaystyle \max_{z \in \mathcal{K}} \left|z^m\right| &\mbox{if} & \kappa=m, \\ & & \\
\displaystyle \frac{1}{\mbox{dist}(\mathcal{K}, \Sigma_N)} &\mbox{if} & \kappa=-1.
\end{array}
\right.
\]
Observe that
\[
\mbox{dist}(\mathcal{K},\Sigma_N) \to \mbox{dist}(\mathcal{K},S_0)=\min_{(z,\zeta)\in \mathcal{K}\times S_0} \left|z-\zeta\right| \quad \mbox{as} \quad n\to \infty.
\]
Considering Lemma \ref{abeycanalla} we decompose the above quotient in simple fractions  
\[
\left| \frac{T_{n-1}(z)}{T_{n}(z)}\right|=M(\mathcal{K}, N)\sum_{k=0}^{m}\sum_{j=1}^d \left|\frac{\lambda_{n,j,k}}{z-\zeta_{n,j}\exp\frac{2 \pi i}{m+1}}\right| 
\]
\[
\leq \frac{M(\mathcal{K}, N)}{\mbox{dist}(\mathcal{K},\Sigma_N)} \sum_{k=0}^{m}\sum_{j=1}^d \left|\lambda_{n,j,k}\right|=\frac{M(\mathcal{K}, N)}{\mbox{dist}(\mathcal{K},\Sigma_N)} \sum_{k=0}^{m}\sum_{j=1}^d \lambda_{n,j,k}
\]
\[
=\frac{M(\mathcal{K}, N)}{\mbox{dist}(\mathcal{K},\Sigma_N)} (m+1) \mu \lambda^m\to \frac{M(\mathcal{K})}{\mbox{dist}(\mathcal{K},S_0)} (m+1) \mu \lambda^m \quad \mbox{as} \quad N\to \infty,
\]
with 
\[
M(\mathcal{K})=\left\{\begin{array}{l c l}
\displaystyle \max_{z \in \mathcal{K}} \left|z^m\right| &\mbox{if} & \kappa=m, \\ & & \\
\displaystyle \frac{1}{\mbox{dist}(\mathcal{K}, S_0)} &\mbox{if} & \kappa=-1.
\end{array}
\right.
\]
This implies that the sequence of quotients $T_{n-1}(z)/T_{n}$, $n\in \mathbb{Z}_+$ is a normal family of holomorphic functions in $\mathcal{K}$, hence each subsequence of it contains a convergent subsequence. We then prove that all convergent subsequence has the limit stated in (\ref{quotienelimit}).

Consider a convergent subsequence  $T_{n^{\prime}-1}(z)/T_{n^{\prime}}$, $n^{\prime} \in \Lambda \subset \mathbb{Z}_+$ with limit
\[
\lim_{n^{\prime} \to \infty} \frac{T_{n^{\prime}-1}(z)}{T_{n^{\prime}}(z)}=\ell(z)
\]
From recurrence relation  (\ref{recurrenceT}) we obtain that
\[
0=1-z \ell(z)+\mu \lambda^m \ell^{m+1} \quad \mbox{or equivelently} \quad 0=\left(\frac{1}{\ell}\right)^{m+1}-z \left(\frac{1}{\ell}\right)^m+\mu \lambda^m.
\]
The function $1/\ell$ is a solution of (\ref{omegaequation}). The branch that admites analytic  extension to $\overline{\mathbb{C}}\setminus S_0$ is $1/\omega_0$.  This completes the proof.
\end{proof}

Let us recall the second kind polynomials in (\ref{secondkindpolLT}). The vector function
\begin{equation}\label{HPfixed}
{\bf R}_n(z)=\left(\frac{T_{n,1}}{T_{n}}, \ldots, \frac{T_{n,m}}{T_n}\right), \quad n\in \mathbb{Z}_+
\end{equation}
is often called Hermite-Pad\'e approximant associated to the system of functions $\displaystyle {\bf f}=\left(f_1,\ldots, f_m\right)$ and the multi-index $\displaystyle {\bf n}=\left(\kappa+1,\ldots, \kappa+1,\kappa,\ldots,\kappa\right)\in \mathbb{Z}_+^m$ with $\ell^1$-norm $|{\bf n}|=n=m\,\kappa+s$.  In  \cite{Ber} we find that O. Perron developed a formulation of Hermite-Pad\'e approximation in terms of generalized continued fractions. Later on V. I Parusnikov \cite{P} considered a called Jacobi-Perron's algorithm based on this formulation. Finally V. Kalyaguin, in several publications such as \cite{Ka, Ka1} generalized such algorithm to extend Favar's theorem to the field of multi-orthogonality of polynomials. Since our problems deals with an easier particular case, we arrive at explicit expressions of the resolvent functions via corresponding sequences of Hermite-Pad\'e approximants.

We revisite (\ref{shift}) in Lemma \ref{motorina}. So  each component of $\displaystyle {\bf R}_n=\left(\frac{T_{n-1}}{T_{n}}, \ldots, \frac{T_{n-m}}{T_n}\right)$ can be written as
\[
\frac{T_{n,j}}{T_n}=\frac{T_{n-1}}{T_n}\frac{T_{n-2}}{T_{n-1}}\cdots \frac{T_{n-j}}{T_{n_j+1}}, \quad j\in \{1,\ldots,m\}.
\]
From Lemma \ref{propcociente} we obtain the following limits uniformly
\begin{equation}\label{markovquotient}
\lim_{n\to \infty} \frac{T_{n,j}}{T_{n}}(z)=\frac{1}{\omega_0^{j}(z)}, \quad  j\in \{1,\ldots,m\}, \quad \mathcal{K}\subset \overline{\mathbb{C}} \setminus S_0,
\end{equation}
with $\mathcal{K}$ being compact.

\begin{proposition}\label{semicorona} Set the resolvent functions $f_j(z, {\bf T}),$ $j=1,\ldots,m$ as in (\ref{resolvente}) corresponding to the operator ${\bf T}$. Then 
\begin{equation}\label{byebye}
f_j(z, {\bf T})=\frac{1}{\omega_0^{j}(z)}, \quad  j\in \{1,\ldots,m\}, 
\end{equation}
in every compact $ \mathcal{K}\subset \overline{\mathbb{C}} \setminus S_0.$
\end{proposition}

\begin{proof} We already have (\ref{markovquotient}), then we only need to prove that the components of Hermite-Pad\'e approximants $T_{n,j}/T_n$ converge to their corresponding resolvent functions $f_j$. Let us recall Lemma \ref{visual} 
\begin{equation}\label{Markov1}
f_j(z, {\bf T})-\frac{T_{n,j}(z)}{T_n(z)}=\mathcal{O}\left(\frac{1}{z^{n+\tau(n,j)}}\right) \quad \mbox{as}\quad z\to \infty,
\end{equation}
with
\[
\tau(n,j)=\left\{\begin{array}{l l l}
\kappa+m+j-s+1 & \mbox{if} & j\in \{1,\ldots, m\} \\ & & \\
\kappa+j-s+1 & \mbox{if} & j\in \{s+1,\ldots,m\}.
\end{array}
\right.
\]
Fix $N$ and $\Sigma_N$ the smallest starlike set that contains all the zeros of the polynomials $T_n$ with $n \geq N$. Then
\begin{equation}\label{Markov2}
f_j(z, {\bf T})-\frac{T_{n,j}(z)}{T_n(z)} \in \mathcal{H}\left(\overline{\mathbb{C}}\setminus \Sigma_N)\right),
\end{equation}
Combining (\ref{Markov1}) and (\ref{Markov2}) we obtain
\[
\left(f_j(\cdot, {\bf T})-\frac{T_{n,j}}{T_n}\right)\omega_0^{n+\tau(n,j)} \in \mathcal{H} \left(\overline{\mathbb{C}}\setminus \Sigma_N\right).
\]
Let us consider $\alpha>1$ such that the level curve
\[
\gamma_{\alpha}=\left\{z \in \mathbb{C}\setminus S_0: \left|\frac{1}{\omega_0(z)}\right|=\alpha\right\}\subset \mathbb{C}\setminus \Sigma_N.
\]
Note that $\gamma_{\alpha}$ surrounds $S_0 \subset \Sigma_N$. We take a compact set $\mathcal{K}\subset \overline{\mathbb{C}} \setminus \Sigma_N$ lying in the unbounded domain (connected and open region) whose boundary is $\gamma_{\alpha}$. Using maximum principle for $1/\omega_0$, the fact that $f_j(\cdot, {\bf T})$ is analytic on $\mathcal{K}$, which yields there exists  $C>0$ such that $|f_j(z, {\bf T})|\leq C$, $z\in \mathcal{K}$, and Lemma \ref{romel}, we obtain that
\[
\left|\left|\left(f_j(z, {\bf T})-\frac{T_{n,j}(z)}{T_n(z)}\right)\right|\right|_{\mathcal{K}}\leq \frac{C+(m+1) \mu \lambda^m}{\mbox{dist}(\Sigma_N,\gamma_{\alpha})}\left(\frac{1}{\displaystyle \alpha \min_{z\in \mathcal{K}}\left|\omega_0(z)\right|}\right)^{n+\tau(n,j)}.
\] 
Since $\displaystyle \alpha  \min_{z\in \mathcal{K}}\left|\omega_0(z)\right|>1$, we obtain that
\[
\limsup_{n \to \infty} \left|\left|\left(f_j(z)-\frac{T_{n,j}(z)}{T_n(z)}\right)\right|\right|_{\mathcal{K}}^{1/(n+\kappa)}=0.
\]
This proves the equality (\ref{byebye}) in $\overline{\mathbb{C}}\setminus \Sigma_N$. We now take $N \to \infty$ and complete the proof. 
\end{proof}

Once we have arrived to the equalities in  (\ref{byebye}) finding the integral expressions for the resolvent functions $f_j(\cdot, {\bf T}),$ $j=1,2,\ldots, m$, is equivalent to obtaining integral expressions for the functions $\displaystyle \frac{1}{\omega_0^j}$, $j=1,2,\ldots,m$. 

\begin{lemma}\label{weight0} There exists a function $\widetilde{\rho}$ defined on $S_0$ such that
\begin{equation}\label{rhosymmetry}
\widetilde{\rho}(x)=\exp\left({-\frac{2 \pi i k}{m+1}}\right)\widetilde{\rho}(|x|), \quad x \in  \left(0, a\exp{\frac{2 \pi i k}{m+1}}\right), 
\end{equation}
$k=0,1,\ldots,m$, where $\widetilde{\rho}(x)$ is a weight in $x\in (0,a))$, that satisfies 
\begin{equation}\label{markovfunctionproved1}
\displaystyle \frac{1}{\omega_0(z)}=\frac{1}{2 \pi i}\int  \widetilde{\rho}(x)\, \frac{ d\, x}{z-x}=\frac{1}{2 \pi i}\int  \widetilde{\rho}(|x|)\, \frac{ d\, |x|}{z-x}, \quad z \in \overline{\mathbb{C}}\setminus S_0.
\end{equation}
\end{lemma}

\begin{proof} Let us denote 
\[
 \omega_{0\pm}(x)=\lim_{z\to x\in \overset{\circ}{S}_{0\pm}}  \omega_0(z) \quad \mbox{with} \quad \overset{\circ}{S}_0=\bigcup_{k=0}^m \left(0, a\exp{\frac{2 \pi i k}{m+1}}\right). 
\]
Consider an arbitrary point $x \in (0,\, a) \subset S_0$ in the star's arm corresponding to $k=0$. Using (\ref{inequalitycomplex}) we observe that the following functions defined on $\overset{\circ}{S}_0$, are real valued and never change sign
\[
\left[\frac{1}{\omega_{0+}(x)}-\frac{1}{\omega_{0-}(x)}\right]= \frac{1}{\omega_{0}(x)}+\frac{1}{\omega_{1}(x)}=\widetilde{\rho}(x), \quad x \in  (0,\,a).
\]
By symmetry we obtain that 
\begin{equation}\label{rjequation1}
\widetilde{\rho}(x)=\exp\left({-\frac{2 \pi i k}{m+1}}\right)\widetilde{\rho}(|x|)=\left[\frac{1}{\omega_{0+}(x)}-\frac{1}{\omega_{0-}(x)}\right],
\end{equation} 
for all 
\[
x \in  \left(0, a\exp{\frac{2 \pi i k}{m+1}}\right), \quad k\in \{0,1,\ldots, m\}.
\]
This prove condition (\ref{rhosymmetry}) holds. Since $\omega_0\in \mathcal{H}(\Omega_0)$ never vanishes in $\displaystyle \Omega_0$ and taking into account (\ref{maltrecho}) we have that
\begin{equation}\label{recua1}
\frac{1}{\omega_0(z)} \in \mathcal{H}(\Omega_0) \quad \mbox{and}  \quad \frac{1}{\omega_0(z)}=\mathcal{O}\left(\frac{1}{z}\right) \quad \mbox{as} \quad z\to \infty.
\end{equation}
Now using (\ref{rjequation1}) and  the Sokhotski–Plemelj formula (see \cite[Chapter I, equality (4.8)]{G}. According to \cite{Fo} the function $1/\omega_0$ behaves properly at the endpoints to apply the Sokhotski–Plemelj formula) we obtain that
\[
\displaystyle \frac{1}{\omega_0(z)}=\frac{1}{2 \pi i}\int  \widetilde{\rho}(x)\frac{ d\, x}{z-x},
\]
that combined to (\ref{rjequation1}), completes the proof of (\ref{markovfunctionproved1}).
\end{proof}

We are ready to give integral expressions for the functions $\displaystyle \left(\frac{1}{\omega_0}\right)^{j},$ $j=1,\ldots,m$. 

\begin{proposition}\label{intergagralforms} There exists a system of weights $r_j(x)=x^{j-1}\widetilde{\rho}_{j}(x)$, $j=1, \ldots, m$, defined on $S_0$ that satisfy the following identities
\begin{equation}\label{markovfunction}
\displaystyle \frac{1}{\omega_0^{j}(z)}=\frac{1}{z^{j-1}}\int x^{j-1}\widetilde{\rho}_{j}(x)\, \widetilde{\rho}(x)\frac{ d\, x}{z-x}, \quad j=1,\ldots,m,  \quad z \in \overline{\mathbb{C}}\setminus S_0,
\end{equation}
where $\widetilde{\rho}_1 \equiv 1$ and $\widetilde{\rho}$ is defined in Lemma \ref{weight0}.
\end{proposition}

 \begin{proof} The case $j=1$ in (\ref{markovfunction}) is already proved in Lemma \ref{weight0}. For each $j\in \{2, \ldots, m\}$, we have that
\[
\frac{1}{\omega_{0+}^j(x)}-\frac{1}{\omega_{0-}^j(x)}=\widetilde{\rho}(x) \sum_{k=0}^{j-1} \frac{1}{\omega_{0+}^{j-1-k}(x)\,\omega_{0-}^{k}(x)}, \quad x \in S_0.
\]
Observe that the following functions are weights
\[
 r_j(x)=x^{j-1} \widetilde{\rho}_j(x)= \sum_{k=0}^{j-1} \frac{x^{j-1}}{\omega_{0+}^{j-1-k}(x)\,\omega_{0-}^{k}(x)}\in \mathbb{R}_+, \quad x \in S_0, \quad j=1,\ldots,m.
\]
From (\ref{recua1}) we obtain that
\begin{equation}\label{recua}
\frac{z^{j-1}}{\omega_0^{j}(z)} \in \mathcal{H}(\Omega_0) \quad \mbox{and}  \quad \frac{z^{j-1}}{\omega_0^{j}(z)}=\mathcal{O}\left(\frac{1}{z}\right) \quad \mbox{as} \quad z\to \infty, \quad j=2,\ldots,m.
\end{equation}
The equalities in (\ref{markovfunction}) are proved combining (\ref{recua}) and  Sokhotski–Plemelj (\cite[Chapter I, equality (4.8)]{G}). The proof is completed.
\end{proof}

\begin{proposition}\label{trivialorthogonality} Let $\widetilde{\rho}, \widetilde{\rho}_1, \ldots, \widetilde{\rho}_m$ be the system of functions in Proposition \ref{intergagralforms}. Then
\[
0=\int x^{\nu} \rho_j(x)\, \rho(x) \, d x,\quad \nu=0,\ldots, j-2, \quad j=2,\ldots,m.
\]
\end{proposition}

\begin{proof} From construction we have that
\[
\widetilde{\rho}_j(x)=\sum_{k=0}^{j-1} \frac{1}{\omega_{0_+}^{j-1-1-k}(x)\omega_{0-}^k(x)}, \quad x \in S_0.
\]
This means that for each $x\in \mathbb{R} \cap S_0$ we have that for each $\nu\in \{0,\ldots,j-1\}$ and $k\in \{0,1,\ldots,m\}$
\[
x^{\nu}\exp \frac{2 \pi i k \nu}{m+1}\widetilde{\rho}_j\left(x\exp \frac{2 \pi i k}{m+1}\right)=\exp \left(\frac{2 \pi i k (\nu-j+1)}{m+1}\right)x^{\nu}\widetilde{\rho}_j(x),
\]
and
\[
\int_{S_0} x^{\nu} \widetilde{\rho}_j(x)\, \widetilde{\rho}(x) \, d x=\sum_{k=0}^m \left(\exp \frac{-2 \pi i  (j-1-\nu)}{m+1} \right)^k \int_{S_0\cap \mathbb{R}} x^{\nu}  \widetilde{\rho}_j(x)\, \widetilde{\rho}(x) \, d x.
\]
The proof is completed after realization that 
\[
\sum_{k=0}^m \left(\exp \frac{-2 \pi i  (j-1-\nu)}{m+1} \right)^k=0, \quad \nu \in \{0,\ldots, j-2\}.
\]
\end{proof}

Combining Proposition \ref{intergagralforms} and Proposition \ref{semicorona} we deduce the following result

\begin{corollary}\label{markovfreal} 
\begin{equation}\label{markovrealeq}
\displaystyle f_j(z, {\bf T})=\frac{1}{\omega_0^j(z)}=\frac{1}{z^{j-1}}\int  x^{j-1}\, \widetilde{\rho}_{j-1}(x) \, \widetilde{\rho}(x) \,\frac{ d\, x}{z-x}, \quad j=1,\ldots,m,
\end{equation}
in every compact $ \mathcal{K}\subset \overline{\mathbb{C}} \setminus S_0$ and $\widetilde{\rho},\widetilde{\rho}_0, \widetilde{\rho}_1, \ldots, \widetilde{\rho}_{m-1}$ being the functions introduced in Proposition \ref{intergagralforms}.
\end{corollary}

According to \cite{AKS} the functions $\widetilde{\rho}_1, \ldots, \widetilde{\rho}_m$ conforms a Nikishin system. If we now consider the results given in publications \cite{AL1, AL2, AE} about Nikishin systems on starlike sets, we conclude that every polynomial in the sequence $\displaystyle \left\{T_n\right\}_{n\in \mathbb{Z}_+}$ vanishes in the interior of $S_0$, being the set of those zeros dense in the whole $S_0$, hence $\Sigma_N=S_0$, for all $N \in \mathbb{Z}_+$. Certainly these papers require of such Nikishin systems to be generated by measures with bounded supports, however the results on $T_n$'s zeros location are easily extended for the case of unbounded supports, as we procede in \cite{{FL4II}} with respect to the previous \cite{FL4} in the real line.

\begin{corollary}\label{momentT} For each $j\in \{1,2,\ldots,m\}$ the moments corresponding to ${\bf T}$ as in (\ref{momentdefinition}) have integral forms
\[
c_{\nu,j, {\bf T}}=\left({\bf T}\right)^{\nu} {\bf e}_{j-1} \cdot {\bf e}_0=\int x^{\nu} \widetilde{\rho}_j(x)\, \widetilde{\rho}(x) \, d x, \quad \nu \in \mathbb{Z}_+. 
\]
When $\nu. \in \{0,\ldots, j-1\}$, $c_{\nu,j, {\bf T}}=0$.
\end{corollary}

\begin{proof} Fix $j\in \{1,\ldots,m\}$. The first $j-1$th cases are deduced by combining  Proposition \ref{trivialorthogonality} and Lemma \ref{firstmoments}.   Let us recall the definition of the resolvent function in (\ref{resolvente}) corresponding to the operator ${\bf T}$. Explicitly
\[
f_j(z, {\bf T})=\left(z \mathbb{I}-{\bf T}\right)^{-1}{\bf e}_{j-1} \cdot {\bf e}_0, \quad j=1,\ldots,m.
\]
Using Neumann's series for operator inversion to  $\displaystyle \left(z \mathbb{I}-{\bf T}\right)^{-1}$ in the weak $\ell_2$ topology,  we have the equality 
\[
f_j(z, {\bf T})=\sum_{\nu=0}^{\infty} \frac{({\bf T})^{\nu} {\bf e}_{j-1}\cdot {\bf e}_0}{z^{\nu+1}}=\sum_{\nu=0}^{\infty} \frac{c_{\nu,j}}{z^{\nu+1}}=\frac{1}{z^{j-1}}\sum_{\nu=0}^{\infty} \frac{c_{\nu+j-1,j, {\bf T}}}{z^{\nu+1}}
\]
uniformly in any compact set contained in the widest open disk of $\overline{\mathbb{C}}$ centered at $z=\infty$ leaving outside $S_0$. Now using the equality (\ref{markovrealeq}) in Corollary \ref{markovfreal} we obtain that
\[
f_j(z, {\bf T})=\frac{1}{z^{j-1}}\sum_{\nu=0}^{\infty} \frac{c_{\nu+j-1,j, {\bf T}}}{z^{\nu+1}}=\frac{1}{z^{j-1}}\sum_{\nu=0}^{\infty} \frac{1}{z^{\nu+1}} \int x^{\nu+j-1} \widetilde{\rho}_{j-1}(x)\, \widetilde{\rho}(x) \, d x,
\] 
uniformly. This completes the proof.
\end{proof}

\subsection{Integral expressions of $f_j(\cdot, {\bf A})$}

We  deal with the resolvent functions in (\ref{resolvente}) corresponding to the operators ${\bf A}$, ${\bf T}$ and ${\bf L}$. In order to distinguish  them for each  $j\in \{1,\ldots,m\}$ we relabel $f_j(\cdot, {\bf A})$, $f_j(\cdot, {\bf T})$ and $f_j(\cdot, {\bf L})$. Let us recall the starlike $\Sigma_0$ in Theorem \ref{multiQ}
\[
\Sigma_0=\bigcup_{k=0}^m \left[-\lambda-\mu, -\lambda-\mu + \frac{m+1}{m}\left(\frac{\mu \lambda^m}{m}\right)^{1/(m+1)} \exp \frac{2 \pi i k}{m+1}\right].
\]
We also consider the starlike set
\[
V=
\]
\[
=\bigcup_{k=0}^m \left[\left(-\lambda-\mu -\sqrt[m+1]{1+\lambda \mu}\right) \exp \frac{2 \pi i k}{m+1}, \left(-\lambda-\mu +\sqrt[m+1]{1+\lambda \mu}\right) \exp \frac{2 \pi i k}{m+1}\right].
\] 
\begin{lemma}\label{resolventesALlemma} For each $j\in \{1,\ldots,m\}$ there is a set $U\subset V$ where
\begin{equation}\label{resolventeAL}. 
f_j(z,{\bf A})=\frac{1}{\lambda^{j-1}} f_j(z+\lambda+\mu,{\bf L}), \quad z \in \overline{\mathbb{C}}\setminus U.
\end{equation}
\end{lemma}

\begin{proof} Fix $j\in \{1,\ldots, m\}$. We use the identity (\ref{matrixidentity}) of Remark \ref{matrixidentityremark} in (\ref{resolvente}) to obtain
\[
f_j(z, {\bf A})=\left[z \mathbb{I}-{\bf A}\right]^{-1} {\bf e}_{j-1}\cdot {\bf e}_0=\left[z \mathbb{I}-{\bf \Lambda}^{-1}\left({\bf L}-(\lambda+\mu)\right){\bf \Lambda}\right]^{-1} {\bf e}_{j-1}\cdot {\bf e}_0
\]
\[
={\bf \Lambda}\left[\left(z+\lambda+\mu\right) \mathbb{I}-{\bf L} \right]^{-1}{\bf \Lambda}^{-1} {\bf e}_{j-1}\cdot {\bf e}_0=\left[\left(z+\lambda+\mu\right) \mathbb{I}-{\bf L} \right]^{-1}{\bf \Lambda}^{-1} {\bf e}_{j-1}\cdot {\bf \Lambda}^{\top}{\bf e}_0
\]
\[
=\frac{1}{\lambda^{j-1}}\left[\left(z+\lambda+\mu\right) \mathbb{I}-{\bf L}\right]^{-1} {\bf e}_{j-1}\cdot {\bf e}_0.
\]
This completes the proof.
\end{proof}

\begin{lemma}\label{maripol1} For each $j\in \{1,\ldots,m\}$ there is an open dick $\mathcal{V}$ in the Riemann sphere $\overline{\mathbb{C}}$ centred at $z=\infty$  such that $\mathcal{V} \cap \left(\Sigma\cup \{\mu\}\right)=\emptyset$ where on any compact set the following equalities hold  
\begin{equation}\label{maripol}
\displaystyle f_j(z, {\bf L})=\frac{1}{z^{j-1}}\int \frac{x^{j-1}\, \widetilde{\rho}_{j-1}(x)\, \widetilde{\rho}(x)\,  d\, x}{z-x}+\int \frac{d (\delta_{\mu}-\delta_0)(x)}{(z-x)^{j-1}}, 
\end{equation}
with $\widetilde{\rho},\widetilde{\rho}_0, \ldots, \widetilde{\rho}_{m-1}$ being the functions introduced in Proposition \ref{intergagralforms}.
Equivalently:
\begin{equation}\label{maripoldistribution}
\displaystyle f_j(z, {\bf L})=\int \frac{\left(\widetilde{\rho}_{j-1}(x)\, \widetilde{\rho}(x)+ \frac{(-1)^{j}}{(j-2)!} (\delta_{\mu}-\delta_0)^{(j-2)}(x)\right)d x}{z-x}.
\end{equation}
\end{lemma}

\begin{proof}
Note ${\bf L}={\bf T}-\mu \mathbb{I}_{m}$ where the infinite diagonal matrix $\displaystyle \mathbb{I}_m=[e_{i,j}]_{(i,j)\in \mathbb{Z}_+^2}$ has form
\[
\mathbb{I}_m=[{\bf e}_0\, {\bf e}_1\, \cdots {\bf e}_{m-1}\, {\bf 0}]=\left\{\begin{array}{l l l}
e_{j,j}=1 & \mbox{if} & j\in \{0,1,\ldots,m-1\}, \\ & & \\
e_{i,j}=0, & \mbox{otherwise}. &
\end{array}\right.
\] 
Then for each $j\in \{1,2,\ldots, m\}$
\[
f_j(z, {\bf L})={\bf e}_{j-1}\cdot \left[z\mathbb{I}-{\bf L}^{\top}\right]^{-1} {\bf e}_0={\bf e}_j \cdot \left[z\mathbb{I}-{\bf T}^{\top}+\mu \mathbb{I}_m\right]^{-1}{\bf e}_0
\]
\[
={\bf e}_{j-1} \cdot \left[\left(\mathbb{I}+\mu \mathbb{I}_m\left(z\mathbb{I}-{\bf T}^{\top}\right)^{-1}\right)\left(z\mathbb{I}-{\bf T}^{\top}\right)\right]^{-1}{\bf e}_0.
\]
Hence
\begin{equation}\label{fLnice}
f_j(z, {\bf L})={\bf e}_{j-1} \cdot\left[z\mathbb{I}-{\bf T}^{\top}\right]^{-1}\left[\mathbb{I}+\mu \mathbb{I}_m\left(z\mathbb{I}-{\bf T}^{\top}\right)^{-1}\right]^{-1} {\bf e}_0
\end{equation}
The operator 
\[
\mu \mathbb{I}_m\left(z\mathbb{I}-{\bf T}^{\top}\right)^{-1} \to 0 \quad \mbox{as} \quad z \to \infty,
\]
uniformly. Taking into account its support is also bounded we have that there is a neighborhood of $z=\infty$ where the following equalities hold uniformly
\[
\displaystyle \left[\mathbb{I}+\mu \mathbb{I}_m\left(z\mathbb{I}-{\bf T}^{\top}\right)^{-1}\right]^{-1}=\mathbb{I}+\sum_{k=1}^{\infty} \mu^{k} \mathbb{I}_m \left[\left(z\mathbb{I}-{\bf T}^{\top}\right)^{-1}\right]^{k}
\]
\[
=\mathbb{I}+\sum_{k=1}^{\infty} \frac{\mu^{k}}{z^k} \mathbb{I}_m \left[\sum_{l=0}^{\infty}\frac{{\bf T}^{\top l}}{z^{l}}\right]^{k}=\mathbb{I}+\sum_{k=1}^{\infty} \frac{\mu^{k}}{z^k}\mathbb{I}_m \sum_{l=0}^{\infty}{l+k-1\choose k-1}\frac{{\bf T}^{\top l}}{z^{l}}.
\]
Let us consider the product
\[
\left[z\mathbb{I}-{\bf T}^{\top}\right]^{-1}\left[\mathbb{I}+\mu \mathbb{I}_m\left(z\mathbb{I}-{\bf T}^{\top}\right)^{-1}\right]^{-1} \]
\[
=\left[z\mathbb{I}-{\bf T}^{\top}\right]^{-1}\left(\mathbb{I}+\sum_{k=1}^{\infty} \frac{\mu^{k}}{z^k}\mathbb{I}_m \sum_{l=0}^{\infty}{l+k-1\choose k-1}\frac{{\bf T}^{\top l}}{z^{l}}\right)
\]
\[
=\left[z\mathbb{I}-{\bf T}^{\top}\right]^{-1}+\sum_{\tau=0}^{\infty}\sum_{k=1}^{\infty} \frac{\mu^{k}}{z^k}\sum_{l=0}^{\infty}{l+k-1\choose k-1}\frac{{\bf T}^{\top \tau}{\mathbb{I}_m\bf T}^{\top l}}{z^{l+\tau}}.
\]
Fix $j\in \{1,2, \ldots, m\}$, and using (\ref{fLnice}) we obtain that
\[
f_j(z, {\bf L})=f_j(z,{\bf T})+{\bf e}_{j-1} \cdot \sum_{k=1}^{\infty}  \sum_{\ell=0}^{\infty}  \sum_{\tau=0}^{\infty}\frac{\mu^{k}}{z^k}{\bf T}^{\top \tau}\mathbb{I}_m{\ell+k-1\choose k-1}\frac{{\bf T}^{\top \ell}}{z^{\ell+\tau}} {\bf e}_0.
\]
Then we have that
\begin{equation}\label{reduceble}
f_j(z, {\bf L})=f_j(z,{\bf T})+{\bf e}_{j-1} \cdot \sum_{k=1}^{\infty}  \sum_{\ell=0}^{m-1}  \sum_{\tau=0}^{\infty}\frac{\mu^{k}}{z^k}{\ell+k-1\choose k-1}\frac{{\bf T}^{\top \tau}{\bf e}_{\ell}}{z^{\ell+\tau}} .
\end{equation}
Observe that
\[
{\bf e}_{j-1}\cdot {\bf T}^{\tau} {\bf e}_{\ell}={\bf e}_{j-1}\cdot {\bf e}_{\ell+\tau}=\delta_{j-1,\ell+\tau}.
\] 
Hence the expression (\ref{reduceble}) is then reducible to 
\[
f_j(z, {\bf L})= f_j(z, {\bf T})+\frac{1}{z^{j-1}}  \sum_{\tau=0}^{j-1} \sum_{k=1}^{\infty}  {j-1-\tau+k-1\choose k-1} \frac{\mu^{k}}{z^{k}}
\]
\[
= f_j(z, {\bf T})+\frac{1}{z^{j-1}} \sum_{k=1}^{\infty} \frac{\mu^{k}}{z^{k}}  \sum_{\tau=0}^{j-1}  {j-1-\tau+k-1\choose k-1} 
\]
\[
= f_j(z, {\bf T})+\frac{1}{z^{j-1}} \sum_{k=1}^{\infty} \frac{\mu^{k}}{z^{k}}  \sum_{\tau=0}^{j-1}  {\tau+k-1\choose k-1}=f_j(z, {\bf T})+\frac{1}{z^{j-1}} \sum_{k=1}^{\infty} \frac{\mu^{k}}{z^{k}}  \sum_{\tau=0}^{j-1}  {\tau+k-1\choose \tau}
\]
\[
= f_j(z, {\bf T})+\frac{1}{z^{j-1}}\,\sum_{k=1}^{\infty}  {j-1+k\choose k} \frac{\mu^{k}}{z^{k}} = f_j(z, {\bf T})+ \left(\frac{1}{z-\mu}\right)^{j-1}-\frac{1}{z^{j-1}}.
\]
The proof of identity in (\ref{maripol}) are derived from Corollary \ref{markovfreal}. 

In order to prove the equivalence between (\ref{maripol}) and (\ref{maripoldistribution}) we observe that 
\[
\frac{1}{(j-2)!}\frac{\partial^{j-2}}{\partial x^{j-2}} \left(\frac{1}{z-x}\right)=\frac{1}{(z-x)^{j-1}}.
\]
From the definition of the distribution derivative corresponding to $\delta_{\zeta}$ in (\ref{deltaderivative}) we obtain that
\[
\int  \frac{(\delta_{\mu}-\delta_0)(x)\, d x}{(z-x)^{j-1}}= \int  \frac{1}{(j-2)!}\frac{\partial^{j-2}}{\partial x^{j-2}} \left(\frac{1}{z-x}\right)\,(\delta_{\mu}-\delta_0)\, d x
\]
\[
=\frac{(-1)^{j-2}}{(j-2)!} \int  \frac{(\delta_{\mu}-\delta_0)^{(j-2)}(x)\, d x}{z-x}.
\]
This chain of equalities completes the proof.
\end{proof}

For each $j=1,\ldots,m$, the integral form of the resolvent function $f_j(z, {\bf L})$ in (\ref{maripol}) contains a term which is 
\[
f_j (z,{\bf T})=\int \frac{\widetilde{\rho}_{j-1}(x)\, \widetilde{\rho}(x)\, d x} {z-x},
\] 
as in (\ref{markovrealeq}) of Corollary \ref{markovfreal} which is often called the Markov function corresponding to the absolutely continuous measure with respect Lebesgue's whose weight is $x^{j-1}\, \widetilde{\rho}_j(x)\, \widetilde{\rho}(|x|)$. The other term is a rational function with poles at a point $z =\mu$. There is no publication on the behavior of its sequence of Hermite-Padé approximants and their denominators $L_n$, $n\in \mathbb{Z}_+$ for these meromorphic variation of Markov functions corresponding to measures supported on starlike sets. In the case of measures supported on sets of the real line we have several the papers on this topics such as \cite{BL, LMF} where the authors obtain convergence conditions of the Hermite-Pad\'e approximants in the sense of capacity.

From the disk $\mathcal{V}$ in Lemma \ref{maripol1} we construct $\mathcal{U}=\{z \in \overline{\mathbb{C}}: z+\lambda+\mu \in \mathcal{V}\}$. 

\begin{proposition}\label{molina} There exists a system of functions $(\rho, \rho_1,\ldots,\rho_m)$ such that for each $j=1,\ldots,m$, $x^{j-1}\rho_j(x)\rho(|x|)$ conforms a weight on $\Sigma_0$, and such that 
\begin{equation}\label{maripoldistributionA}
\displaystyle f_j(z, {\bf A})=\int \frac{\left(\rho_{j-1}(x)\, \rho(x)+ \frac{(-1)^{j}}{\lambda^{j-1}(j-2)!} (\delta_{-\lambda}-\delta_{-\mu-\lambda})^{(j-2)}(x)\right)d x}{z-x}.
\end{equation}
with $z \in \mathcal{U}$, $j=1,\ldots,m$. 
\end{proposition}

\begin{proof} We combine  Lemma \ref{resolventesALlemma} and Lemma \ref{maripol1}.  Considering (\ref{resolventeAL}) let us  plug $\zeta+\lambda+\mu=z$ in (\ref{maripoldistribution}) to obtain
\[
f_j(\zeta, {\bf A})=\frac{1}{\lambda^{j-1}}\int \frac{\left(\widetilde{\rho}_{j-1}(x)\, \widetilde{\rho}(x)+ \frac{(-1)^{j}}{(j-2)!} (\delta_{\mu}-\delta_0)^{(j-2)}(x)\right)d x}{\zeta+\mu+\lambda-x}.
\]
We complete the proof taking the change of variable in the integral $u=x-\lambda-\mu$, and denoting
\[
\rho(u)=\widetilde{\rho}(u+\lambda+\mu) \quad \mbox{and} \quad \rho_j(u)=\frac{1}{\lambda^{j-1}}\widetilde{\rho}_j(u+\lambda+\mu), \quad j=1,\ldots,m.
\]
In order to avoid to use too many different letters in our notation in the statement we recover $x$ and $z$ instead of $u$ and $\zeta$. 
\end{proof}

\begin{remark} Since $\rho, \rho_1,\ldots, \rho_m$ and $\widetilde{\rho}, \widetilde{\rho}_1,\ldots, \widetilde{\rho}_m$ are related by the change of variables $x=u+\lambda+\mu$, the functions $\rho, \rho_1,\ldots, \rho_m$ can be written in terms of the solutions of the algebraic equation given in (\ref{omegaequationQ}) analogously $\widetilde{\rho}, \widetilde{\rho}_1,\ldots, \widetilde{\rho}_m$ are expressed in terms of the ones corresponding to the equation in (\ref{omegaequation}), Section \ref{algebraicequation}.
\end{remark}

Let us denote as $c_{\nu,j}({\bf T})$ and $c_{\nu,j}({\bf A})$, $\nu \in \mathbb{Z}_+$, $j \in \{1,\ldots, m\}$ the momenta corresponding to the operators ${\bf T}$ and ${\bf A}$, respectively, as in (\ref{momentdefinition})
\begin{equation}\label{defmomentA}
c_{\nu,j}({\bf T})={\bf T}^{\nu} {\bf e}_{j-1} \cdot {\bf e}_0 \quad \mbox{and} \quad c_{\nu,j}({\bf A})={\bf A}^{\nu} {\bf e}_{j-1} \cdot {\bf e}_0,  \quad j=1,\ldots,m.
\end{equation}

\begin{proposition}\label{propmomentA} Set the sequences of momenta $\displaystyle \left\{c_{\nu,j}({\bf A})\right\}_{\nu \in \mathbb{Z}_+}$, $j=1,\ldots,m$. Then 
\begin{equation}\label{momentA}
c_{\nu,j}({\bf A})=\int x^{\nu} \left[\rho (x)\,\rho_{j-1} (x)+ \frac{(-1)^{j}}{\lambda^{j-1}(j-2)!}\left(\delta_{-\lambda}-\delta_{-\lambda-\mu}\right)^{(j-2)} (x)\right]\, d x.
\end{equation}
When $\nu. \in \{0,\ldots, j-1\}$ $c_{\nu,j}({\bf A})=0$.
\end{proposition}

\begin{proof} Fix $j\in \{1,\ldots,m\}$. Combining (\ref{maripoldistributionA}) and Neumann's series for matrix inversion to the operator $\displaystyle \left(z \mathbb{I}-{\bf A}^{\top}\right)^{-1}$ we have that
\[
f_j(z,{\bf A})=\sum_{\nu=0}^{\infty} \frac{c_{\nu,j}({\bf A})}{z^{\nu+1}}
\]
\[
=\sum_{\nu=0}^{\infty} \frac{1}{z^{\nu+1}} \int x^{\nu} \left[\rho (x)\,\rho_{j-1} (x)+ \frac{(-1)^{j}}{\lambda^{j-1}(j-2)!}\left(\delta_{-\lambda}-\delta_{-\lambda-\mu}\right)^{(j-2)} (x)\right]\, d x.
\]
\end{proof}

\section{Proofs of Theorem \ref{multiQ} and Theorem \ref{main}}\label{thproofs}

\subsection{Proof of Theorem \ref{multiQ}}\label{HPD}

Let us recall  the sequence $\displaystyle \left\{{\bf q}_r=\left(q_{0,r},\ldots,q_{m-1,r}\right)\right\}_{r \in \mathbb{Z}_+\cup \{-1\}}$  defined by the recurrence relation (\ref{preccurencelatin}). Set an arbitrary vector polynomial $\displaystyle {\bf p}=\left(p_0,\ldots, p_{m-1}\right)$, let us denote for every system ${\bf U}=\left({\bf u}_0,\ldots, {\bf u}_{m-1}\right)$ with ${\bf u}_j\in \ell_2$, $j=0,\ldots,m-1$ the "dot" product
\[
{\bf p}({\bf A}) \bigodot {\bf U}=\sum_{j=0}^{m-1} p_{j}({\bf A}){\bf u}_j.
\]

\begin{lemma}\label{jumplemmap}  Let $\displaystyle {\bf E}=\left({\bf e}_0,\ldots,{\bf e}_{m-1}\right)$ denote. Then
\[
{\bf e}_r={\bf q}_r({\bf A}) \bigodot {\bf E}, \quad r\in \mathbb{Z}_+=\{0,1,2,\ldots\}.
\]
\end{lemma}

\begin{proof} We observe that  for each $\displaystyle r \in \left\{0,1,\ldots, m-1\right\}$ 
\[
{\bf q}_r({\bf A}) \bigodot {\bf E}=(\overset{r+1}{\overbrace{0,\ldots,0}},1,0,\ldots,0) \bigodot ({\bf e}_0,\ldots,{\bf e}_{m-1})={\bf e}_r.
\]
Let us consider $2 m-1 \geq r\geq m-1$. For each $k=0,\ldots,r-1$, let us assume that $\displaystyle {\bf q}_k({\bf A})\bigodot{\bf E}={\bf e}_k$. By the recurrence relation (\ref{preccurencelatin}), we have that
\[
 {\bf q}_{r}({\bf A})\bigodot {\bf E}=\frac{1}{\mu}\left[\left(\lambda+{\bf A}\right) {\bf q}_{r-m}({\bf A})\bigodot {\bf E}-\lambda {\bf q}_{r-m-1}({\bf A})\bigodot {\bf E}\right]
\]
\[
=\frac{1}{\mu}\left[\lambda {\bf q}_{r-m}({\bf A})\bigodot {\bf E}+{\bf A} {\bf q}_{r-m}({\bf A})\bigodot {\bf E}-\lambda {\bf q}_{r-m-1}({\bf A})\bigodot {\bf E}\right]
\]
\[
=\frac{1}{\mu}\left[\lambda{\bf e}_{r-m}+{\bf A} {\bf e}_{r-m}-\lambda {\bf e}_{r-m-1}\right]
\]
\[
=\frac{1}{\mu}\left[\lambda{\bf e}_{r-m}-\lambda{\bf e}_{r-m}+\lambda {\bf e}_{r-m-1}+\mu {\bf e}_{r}-\lambda {\bf e}_{r-m-1}\right]={\bf e}_r.
\]
We now consider $r \geq 2 m-1$
\[
 {\bf q}_{r}({\bf A})\bigodot {\bf E}=\frac{1}{\mu}\left[\left(\lambda+\mu+{\bf A}\right) {\bf q}_{r-m}({\bf A})\bigodot {\bf E}-\lambda {\bf q}_{r-m-1}({\bf A})\bigodot {\bf E}\right]
\]
\[
=\frac{1}{\mu}\left[(\lambda+\mu) {\bf q}_{r-m}({\bf A})\bigodot {\bf E}+{\bf A} {\bf q}_{r-m}({\bf A})\bigodot {\bf E}-\lambda {\bf q}_{r-m-1}({\bf A})\bigodot {\bf E}\right]
\]
\[
=\frac{1}{\mu}\left[(\mu +\lambda){\bf e}_{r-m}+{\bf A} {\bf e}_{r-m}-\lambda {\bf e}_{r-m-1}\right]
\]
\[
=\frac{1}{\mu}\left[(\mu +\lambda){\bf e}_{r-m}-(\mu +\lambda){\bf e}_{r-m}+\iota {\bf e}_{r-m-1}+\mu {\bf e}_{r}-\lambda {\bf e}_{r-m-1}\right]={\bf e}_r.
\]
The proof is completed by induction.
\end{proof}

In $\ell_2$ we consider the regular inner product. This means that given two sequences ${\bf u}=\left\{u_j\in \mathbb{R}\right\}_{j\in \mathbb{Z}_+}$ and ${\bf v}=\left\{v_j\in \mathbb{R}\right\}_{j\in \mathbb{Z}_+}$ of $\ell^2$
\[
{\bf u} \cdot {\bf v}= \sum_{j=0}^{\infty} u_j v_j={\bf u}^{\top} {\bf v}.
\]
 We think the sequences as infinite column vectors.

\begin{proposition}\label{orthogonalitylemma} Set ${\bf E}=({\bf e}_0,\ldots,{\bf e}_{m-1})$. Given an order pair $(n,r)\in \mathbb{Z}_+^2$, the polynomial $Q_n$ in (\ref{recurrenceQ})  and ${\bf q}_r$ in (\ref{preccurencelatin}) satisfy the orthogonality relations:
\begin{equation}\label{orthogonalityKp}
{\bf q}_r ({\bf A})  \bigodot  {\bf E} \cdot Q_n({\bf A}^{\top}) {\bf e}_0=\delta_{n,r}.
\end{equation}
\end{proposition}

\begin{proof} From Lemma \ref{jumplemma} and Lemma \ref{jumplemmap} we have immediately 
\[
{\bf q}_r ({\bf A})  \bigodot  {\bf E} \cdot Q_n({\bf A}^{\top}) {\bf e}_0={\bf e}_r \cdot {\bf e}_n=\delta_{r,n}.
\]
This proves (\ref{orthogonalityKp}).
\end{proof}

The equality (\ref{orthogonalityQ}) can be deduced combining the second equality in (\ref{defmomentA}), the identity (\ref{orthogonalityKp}) in Proposition \ref{orthogonalitylemma}, and (\ref{momentA}) in Proposition \ref{propmomentA}. This completes the proof of Theorem \ref{multiQ}.

\subsection{Proof of Theorem  \ref{main}}\label{remark0}

In Section \ref{Favard} we prove Theorem \ref{multiQ}. We then prove Theorem \ref{main} assuming  the orthogonality relations in (\ref{orthogonalityQ}). Let us start by showing that the functions $P_{n,r}$, $(n,r)\in \mathbb{Z}_+^2$  in (\ref{Abbey}) are entries of a solution of (\ref{initialvalueproblem}).  The integral expression in (\ref{Abbey}) is reduced to (\ref{orthogonalityQ}) when $t=0$, hence the initial condition of the problem (\ref{initialvalueproblem}) holds. When we take the derivative ${\bf P}^{\prime}=[P^{\prime}_{n,r}]$ we obtain:
\[
P_{n,r}^{\prime}(t)=\int e^{x t} x\, Q_n(x) \,\sum_{j=0}^{m-1} q_{j,r}(x) \, d \sigma_j(x).
\]
Using the recurrence relations that define $Q_n$ (\ref{recurrenceQ}) with initial conditions (\ref{recurrenceQv2}) in Section \ref{introduction}, we arrive at
\[
P_{i,j}^{\prime} (t)=\left\{\begin{array}{l}
-\lambda P_{i,j}(t)+\lambda P_{i+1,j}(t), \quad (i,j) \in   \left\{0,1,\ldots, m-1\right\} \times \mathbb{Z}_+=\mathbb{F},\\  \\
-(\lambda+\mu)P_{i,j}(t)+\lambda P_{i+1, j}(t)+ \mu P_{i-m, j}(t), \quad (i,j) \in \mathbb{Z}_+^2 \setminus \mathbb{F},
\end{array}
\right.,
\]
which are the relations in (\ref{recurrencePro}). This proves that $P_{n,r}$, $(n,r)\in \mathbb{Z}_+^2$ as in (\ref{Abbey}) are entries of a matrix ${\bf P}$ which satisfies (\ref{initialvalueproblem}).

We observe that $P_{n,r}$, $(n,r)\in \mathbb{Z}_+^2$  in (\ref{Abbey}) have analytic continuations to the whole complex plane $\mathbb{C}$. The rest of this Section concerns with showing that the initial value problem (\ref{initialvalueproblem}) has a unique solution with analytic functions as entries, and they must be real on $\mathbb{R}_+$. We also show that any other solution does not contain transition probability functions in the entries. Finally we show that $P_{n,r}$, $(n,r)\in \mathbb{Z}_+^2$ satisfy the four properties of a transition probability functions sated in Remark \ref{fourconditions}. 

Let us consider the initial value problem (\ref{initialvalueproblem}). Observe that for each column  of ${\bf P}$ its entries satisfy the same system of equations, hence in order to simplify the notation for this Subsection, we omit the dependence in $j \in \mathbb{Z}_+$. Set the infinite column vector function
\[
{\bf X}(x)=\left(x_0(x),\, x_1(x), \ldots \right)^{\top}=\left(P_{0,j}(x),P_{1,j}(x), \ldots \right)^{\top}.
\] 
Then ${\bf X}$ satisfies an initial value problem:

\begin{equation}\label{Xinitialvalue}
\left\{
\begin{array}{l l}
{\bf X}^{\prime}(t)= {\bf A}{\bf X}(t), & t\geq 0,  \\ & \\
{\bf X}(0)={\bf e}_j,
\end{array}
\right. 
\end{equation}
 where  ${\bf e}_j$ is a column vector of the identity matrix. The following statement is an immediate consequence of the results in \cite{S}. 

\begin{proposition}\label{unireal} There is a unique solution $\displaystyle {\bf X}(t)=\left(x_0(t),\, x_1(t), \ldots \right)^{\top}$ of the initial value problem  (\ref{Xinitialvalue}) that satisfies the two following equivalent conditions
\begin{enumerate}
\item[]
\item The set of entries functions $\displaystyle \left\{x_0(t),\, x_1(t), \ldots \right\}$ is uniformly bounded in any compact disk centered at the origin with radius $r\in \mathbb{R}_+$,  $\displaystyle D_r=\left\{z \in \mathbb{C}: |z|\leq r\right\}$.
\item[]
\item The functions in $\displaystyle \left\{x_0(t),\, x_1(t), \ldots \right\}$ are analytic on $\mathbb{C}$ ($x_j\in \mathcal{H}(\mathbb{C})$, $j\in \mathbb{Z}_+$).
\item[]
\end{enumerate}
This solution $\displaystyle {\bf X}$ has real values function components on $\mathbb{R}_+$ ($x_j: \mathbb{R}_+\rightarrow \mathbb{R}$, $j\in \mathbb{Z}_+$).
 \end{proposition}
 
 \begin{remark}
Since the functions in (\ref{Abbey}) are analytic on $\mathbb{C}$, this result implies that they are the only analytic solution stated in Theorem \ref{main}. 
 \end{remark}

We now prove the four properties. Given an ordered pair $(n,r)\in \mathbb{Z}_+^2$, we now prove that the function $P_{n,r}$ expressed with the integral form (\ref{Abbey}) in Theorem  \ref{main} satisfies the four properties enumerated in Remark \ref{fourconditions}.  Previously we introduce the infinite matrices  $\displaystyle {\bf F}=[f_{i,j}]_{(i,j)\in \mathbb{Z}_+^2}$ and   $\displaystyle {\bf G}=[g_{i,j}]_{(i,j)\in \mathbb{Z}_+^2}$ with forms
\[
f_{i,j}=\left\{\begin{array}{l l}
1, & \mbox{if} \quad   j=i+1,\\  & \\
0,& \mbox{otherwise},
\end{array}
\right.\quad \mbox{and}  \quad g_{i,j}=\left\{\begin{array}{l l}
\displaystyle 1, & \mbox{if} \quad  i=j+m, \\  & \\
0,& \mbox{otherwise},
\end{array}
\right. \quad (i,j)\in \mathbb{Z}_+^2.
\]
Fixed $(a,b)\in \mathbb{Z}_+^2$ let us denote the matrix ${\bf F}^a{\bf G}^b=\mathbb{E}_{a,b}=[e_{i,j,a,b}]_{(i,j)\in \mathbb{Z}_+^2}$, which has  form
\[
e_{i,j,a,b}=\left\{\begin{array}{l l}
1, & \mbox{if} \quad   i=j+am-b,\\  & \\
0,& \mbox{otherwise},
\end{array}
\right. \quad  (i,j)\in \mathbb{Z}_+^2.
\]
We accept that ${\bf E}_{0,0}=\mathbb{I}$. Consider the matrix ${\bf T}=[t_{i,j}]_{(i,j)\in \mathbb{Z}_+^2}$ whose components are defined by
\[
t_{i,j}=\left\{\begin{array}{l l}
\lambda, & \mbox{if} \quad   j=i+1,\\  & \\
\displaystyle \mu, & \mbox{if} \quad  i=j+m, \\  & \\
0,& \mbox{otherwise}.
\end{array}
\right.\quad (i,j)\in \mathbb{Z}_+^2.
\]
Observe that ${\bf T}=\lambda {\bf F}+\mu {\bf G}$ and
\[
{\bf T}^{\nu}=\sum_{\ell=0}^{\nu} {\nu \choose \ell} \lambda^{\ell} \mu^{\nu-\ell}  {\bf F}^{\ell}   {\bf G}^{\nu-\ell}.
\]
We also analyze the infinite matrix $\displaystyle {\bf K}=[k_{i,j}]_{(i,j)\in \mathbb{Z}_+^2}=\mu \mathbb{I}_m+ {\bf T}$, whose entries explicitly are
\[
k_{i,j}=\left\{\begin{array}{l l}
\mu, & \mbox{if} \quad   i=j, \quad 0\leq i\leq m-1\\  & \\
\lambda, & \mbox{if} \quad   j=i+1,\\  & \\
\mu, & \mbox{if} \quad  i=j+m, \\  & \\
0,& \mbox{otherwise} ,
\end{array}
\right.\quad (i,j)\in \mathbb{Z}_+^2,
\]
with form
\[
{\bf K}^{\top}=\left(
\begin{array}{c c c c c c c c c c c c c}
\mu & 0  & 0  &\cdots  & 0 &\mu & 0        & 0  & \cdots &0 & 0 & \cdots\\ 
\lambda & \mu          & 0  & \cdots  & 0 & 0      & \mu & 0  & \cdots & 0 & 0 & \\
0 & \lambda         & \mu  & \cdots  & 0 & 0      &  0  & \mu  &     \ddots     &0  & 0 & \\
0 & 0          & \lambda & \cdots  & 0 & 0      &  0  & 0  &           & 0 & 0 \\
\vdots  &     &  \ddots  &             &  \vdots  &         &  \vdots        &          &    \ddots    &  \ddots & \\
0& 0          & \cdots  & \lambda  & \mu & 0      &  0  & 0  &  & 0 & \mu & \\
0& 0          & \cdots  & 0  & \lambda & 0      &  0  & 0  & \cdots & 0 & 0  & \ddots \\
0& 0          & \cdots  & 0  & 0 & \lambda      &  0  & 0  & \cdots & 0 & 0   & \ddots \\
\vdots  &     &        &  \vdots       & \ddots    &      \ddots   &    \ddots       &     \ddots     &          & \vdots & \vdots & \\
0& 0          & \cdots  & 0   &\cdots  &    0   & 0  &  \lambda & 0 & 0 & 0 &  \\
0& 0          & \cdots  &  0 &\cdots  &      0 &  0 &  0 & \lambda & 0 & 0  &\\
\vdots &           &   &  \vdots &   &       & \ddots &   & \ddots & \ddots &  &  \ddots
\end{array}
\right).
\] 
We rewrite the equalities  (\ref{recurrenceQ}), (\ref{recurrenceQv2})  in a formal matrix way as follows
\[
{\bf K}\, {\bf Q}(x)=(x+\lambda+\mu){\bf Q}(x), \quad  \mbox{where} \quad {\bf Q}=\left(\begin{array}{c} 
Q_0 \\
Q_1\\
\vdots
\end{array}\right).
\]
For each $k \in \mathbb{Z}_+$, we have that $ \displaystyle {\bf K}^k\, {\bf Q}(x)=(x+\lambda+\mu)^k{\bf Q}(x)$, which is
\[
(x+\lambda+\mu)^k{\bf Q}(x)=\left(\mu \mathbb{I}_m+ {\bf T}\right)^k{\bf Q}(x)=\left( {\bf T}^k+\sum_{j=1}^k {k \choose j}\mu^j \mathbb{I}_m {\bf T}^{k-j}\right){\bf Q}(x)
\]
\[
=\left(\sum_{\ell=0}^k {k \choose \ell} \lambda^{\ell} \mu^{k-\ell} {\bf F}^{\ell} {\bf G}^{k-\ell} +\sum_{j=1}^k \sum_{\ell=0}^{k-j}{k \choose j}{k-j \choose \ell} \lambda^{\ell} \mu^{k-\ell} \mathbb{I}_m {\bf F}^{\ell}{\bf G}^{k-j-\ell}\right){\bf Q}(x).
\]
Let $\Phi_m:\mathbb{R} \rightarrow \{0,1\}$ denote the Heaviside function 
\[
\Phi_m(x)=\left\{\begin{array}{l l}
0 & \mbox{if} \quad x< m, \\ &  \\
1 & \mbox{if} \quad x \geq m.
\end{array}
\right.
\]
We also denote   the ceiling function by $\displaystyle \lceil x \rceil=\min\left \{c \in \mathbb{Z}: c \geq x\right\}$.

\begin{lemma}\label{cohen} Fix $(n,k)\in \mathbb{Z}_+^2$. Then
\begin{equation}\label{cupon}
(x+\mu+\lambda)^{k} Q_n(x)=\sum_{j=0}^k {k \choose j} \lambda^j\mu^{k-j} Q_{g(n,k,j)}(x)
\end{equation}
\[
= \sum_{\ell=\max\{0, \lceil\frac{km-n}{m+1}\rceil\}}^k {k \choose \ell} \lambda^{\ell} \mu^{k-\ell} Q_{n+\ell-m(k-\ell)}(x)
\]
\[
+\Phi_m(n)\sum_{j=1}^k \sum_{\ell=\max\{0, \lceil\frac{(k-j)m-n}{m+1}\rceil\}}^{k-j}{k \choose j}{k-j \choose \ell} \lambda^{\ell} \mu^{k-\ell} Q_{n+\ell-m(k-j-\ell)}
\]
\end{lemma}

\begin{proof} Since $\displaystyle {\bf e}_n^{\top} {\bf Q}(x)=Q_n(x)$, $n\in \mathbb{Z}_+$ we have that
\[
(x+\lambda+\mu)^k Q_n(x)
\]
\[
={\bf e}_n^{\top} \left(\sum_{\ell=0}^k {k \choose \ell} \lambda^{\ell} \mu^{k-\ell} {\bf F}^{\ell} {\bf G}^{k-\ell} +\sum_{j=1}^k \sum_{\ell=0}^{k-j}{k \choose j}{k-j \choose \ell} \lambda^{\ell} \mu^{k-\ell} \mathbb{I}_m {\bf F}^{\ell}{\bf G}^{k-j-\ell}\right){\bf Q}(x)
\]
\[
= \sum_{\ell=0}^k {k \choose \ell} \lambda^{\ell} \mu^{k-\ell} {\bf e}_n^{\top} {\bf F}^{\ell} {\bf G}^{k-\ell} {\bf Q}(x)
\]
\[
+\sum_{j=1}^k \sum_{\ell=0}^{k-j}{k \choose j}{k-j \choose \ell} \lambda^{\ell} \mu^{k-\ell} {\bf e}_n^{\top} \mathbb{I}_m {\bf F}^{\ell}{\bf G}^{k-j-\ell}{\bf Q}(x)
\]
\[
= \sum_{\ell=\max\{0, \lceil\frac{km-n}{m+1}\rceil\}}^k {k \choose \ell} \lambda^{\ell} \mu^{k-\ell} {\bf e}_{n+\ell-m(k-\ell)}^{\top}  {\bf Q}(x)
\]
\[
+\sum_{j=1}^k \sum_{\ell=\max\{0, \lceil\frac{(k-j)m-n}{m+1}\rceil\}}^{k-j}{k \choose j}{k-j \choose \ell} \lambda^{\ell} \mu^{k-\ell} \Phi_m(n){\bf e}_{n+\ell-m(k-j-\ell)}^{\top} {\bf Q}(x)
\]
\[
= \sum_{\ell=\max\{0, \lceil\frac{km-n}{m+1}\rceil\}}^k {k \choose \ell} \lambda^{\ell} \mu^{k-\ell} Q_{n+\ell-m(k-\ell)}(x)
\]
\[
+\Phi_m(n)\sum_{j=1}^k \sum_{\ell=\max\{0, \lceil\frac{(k-j)m-n}{m+1}\rceil\}}^{k-j}{k \choose j}{k-j \choose \ell} \lambda^{\ell} \mu^{k-\ell} Q_{n+\ell-m(k-j-\ell)},
\]
which completes the proof.
\end{proof}

Fix $r \in \mathbb{Z}_+$ and consider the vector polynomial $\displaystyle {\bf q}_r=\left(q_{0,r}, \cdots, q_{m-1,r}\right)$ defined by the recurrence relation (\ref{preccurencelatin}). We introduce a distribution form as follows
\[
\mathcal{L}_r(x)=\sum_{j=0}^{m-1} q_{j,r}(x)\, \frac{d \sigma_j(x)}{d x},
\] 
where $\displaystyle \frac{d \sigma_j(x)}{d x}$, $j=0,\ldots,m-1$, were introduced in (\ref{RNdistribution}) of Theorem \ref{multiQ}. Set the sequence of distribution forms ${\bf \mathcal{L}}=\left(\mathcal{L}_0,\mathcal{L}_1, \mathcal{L}_2, \ldots\right)$. The following result is an immediate consequence of the assumed orthogonality relations in (\ref{orthogonalityQ}) and Lemma \ref{cohen}.

\begin{lemma}\label{identityrecurrenceorthgonality}  Let $(n,r,k)\in \mathbb{Z}_+^3$ be given. Then the following equality holds
\begin{equation}\label{fredo}
 \int  (x+\mu+\lambda)^{k} Q_n(x) \mathcal{L}_r(x)\, dx 
  \end{equation}
 \[
 = \sum_{\ell=\max\{0, \lceil\frac{km-n}{m+1}\rceil\}}^k {k \choose \ell} \lambda^{\ell} \mu^{k-\ell} \delta_{n+\ell-m(k-\ell),r}
 \]
\[
+\Phi_m(n)\sum_{j=1}^k \sum_{\ell=\max\{0, \lceil\frac{(k-j)m-n}{m+1}\rceil\}}^{k-j}{k \choose j}{k-j \choose \ell} \lambda^{\ell} \mu^{k-\ell} \delta_{n+\ell-m(k-j-\ell),r}.
\]
\end{lemma}

We are now ready to prove  the first property.

\begin{proposition}\label{property1}  Let $P_{n,r}$, $(n,r)\in \mathbb{Z}_+^2$ be a function defined as in (\ref{Abbey}). Then $P_{n,r}(t)\geq 0$, $t \in \mathbb{R}_+$. Specifically $P_{n,r}(0)=\delta_{n,r}$ and $P_{n,r}(t)> 0$, if $t>0$.
\end{proposition}

\begin{remark}\label{irreductibleproof}  Recall that $P_{n,r}(t)> 0$, $(n,r)\in \mathbb{Z}_+^2$,  $t>0$ defines that $\mathbb{Z}_+$ is irreductible with respect to $P_{n,r}$.
\end{remark}

\begin{proof} In fact $P_{n,r}(0)=\delta_{n,r}$ is the initial condition of the problem  (\ref{initialvalueproblem}) that we have already proved that $P_{n,r}$, $(n,r)\in \mathbb{Z}_+^2$ is a solution. Let us consider $t>0$. From (\ref{Abbey}) we have that
\[
P_{n,r}(t)=\int e^{x t}\, Q_n(x) \mathcal{L}_r(x)\, d x=e^{-(\lambda+\mu)t}\, \int e^{(\lambda +\mu+x) t}\, Q_n(x) \mathcal{L}_r(x) \, d x.
\]
Since the function $e^{-(\lambda+\mu)t}$ is always positive we study the expression
\[
\int e^{(\lambda +\mu+x) t}\, Q_n(x)\, \mathcal{L}_r(x)\, d x=\sum_{j=0}^{\infty} \frac{t^j}{j!} \left[\int (\lambda +\mu+x)^j\, Q_n(x)\, \mathcal{L}_r(x)\, d x\right].
\]
For each $j\in \mathbb{Z}_+$ and $t\in \mathbb{R}_+$, the factor $\displaystyle \frac{t^j}{j!} \in \mathbb{R}_+$. Now taking into account the identity (\ref{fredo}) in Lemma \ref{identityrecurrenceorthgonality}, we see that the factor 
\[
\int (\lambda +\mu+x)^j\, Q_n(x) \mathcal{L}_r(x)\, d x
\]
must be nonnegative. Even more, most of the are positive. This completes the proof and Property (1) in in Remark \ref{fourconditions} is satisfied, and that $\mathbb{Z}_+$ is irreductible.
\end{proof}

Let us prove the second property. This is the honesty of  ${\bf P}$.

\begin{proposition}\label{honesty} Let $\displaystyle {\bf P}=\left[P_{n,r}\right]_{(n,r)\in \mathbb{Z}_+^2}$ be a solution of the initial value problem in (\ref{initialvalueproblem}). Then
\[
\sum_{r=0}^{\infty} P_{n,r}(t)\equiv 1, \quad t\in \mathbb{R}_+, \quad n\in \mathbb{Z}_+. 
\]
\end{proposition}
 
 \begin{proof}
Differentiating $k$th, $k\in \mathbb{Z}_+$ times the relations (\ref{recurrencePro}), and evaluated at $t=0$ we have that 
\[
P_{n,r}^{(k+1)} (0)=\left\{\begin{array}{l}
-\lambda P_{n,r}^{(k)}(0)+\lambda P_{n+1,r}^{(k)}(0), \quad (n,r) \in \left\{0,1,\ldots, m-1\right\}\times  \mathbb{Z}_+=\mathbb{F},\\  \\
-(\lambda+\mu)P_{n,r}^{(k)}(0)+\lambda P_{n+1,r}^{(k)}(0)+ \mu P_{n-m,r}^{(k)}(0), \quad (n,r) \in \mathbb{Z}_+^2 \setminus \mathbb{F}.
\end{array}
\right.
\]
Observe that  
\begin{equation}\label{bobo}
 \sum_{r=0}^{\infty}P_{n,r}^{(k+1)}(0)=
\end{equation}
\[
=\left\{\begin{array}{l}
\displaystyle -\lambda \sum_{r=0}^{\infty} P_{n,r}^{(k)}(0)+\lambda \sum_{r=0}^{\infty} P_{n+1,r}^{(k)}(0), \quad (n,r) \in \left\{0,1,\ldots, m-1\right\}\times  \mathbb{Z}_+=\mathbb{F},\\  \\
\displaystyle -(\lambda+\mu) \sum_{r=0}^{\infty}P_{n,r}^{(k)}(0)+\lambda \sum_{r=0}^{\infty} P_{n+1,r}^{(k)}(0)+ \mu \sum_{r=0}^{\infty} P_{n-m,r}^{(k)}(0), \quad (n,r) \in \mathbb{Z}_+^2 \setminus \mathbb{F},
\end{array}
\right.
\]
hence if we have that $k\in \mathbb{N}$
\[
\sum_{r=0}^{\infty}P_{n,r}^{(k)}(0)=0 \implies  \sum_{r=0}^{\infty}P_{n,r}^{(k+1)}(0)=0.
\]
We only need to prove that  $\displaystyle \sum_{r=0}^{\infty}P_{n,r}^{\prime}(0)=0$ to obtain that $\displaystyle \sum_{r=0}^{\infty}P_{n,r}^{(k)}(0)=0$, $\forall k\in \mathbb{N}$.

From the initial condition we have that 
\begin{equation}\label{complementaria}
P_{n,r}(0)=\delta_{n,r}, \implies \sum_{r=0}^{\infty} P_{n,r}(0)= 1.
\end{equation}
When $k=0$ we have that $\displaystyle \sum_{r=0}^{\infty}P_{n,r}^{\prime}(0)=$
\[
=\left\{\begin{array}{l}
\displaystyle -\lambda \sum_{r=0}^{\infty} P_{n,r}(0)+\lambda \sum_{r=0}^{\infty} P_{n-1,r}(0), \quad (n,r) \in \left\{0,1,\ldots, m-1\right\}\times  \mathbb{Z}_+=\mathbb{F},\\  \\
\displaystyle -(\lambda+\mu) \sum_{r=0}^{\infty}P_{n,r}(0)+\lambda \sum_{r=0}^{\infty} P_{n-1,r}(0)+ \mu \sum_{r=0}^{\infty} P_{n-m,r}(0), \quad (n,r) \in \mathbb{Z}_+^2 \setminus \mathbb{F}.
\end{array}
\right.
\]
Using (\ref{complementaria}) we obtain  that
\[
\sum_{r=0}^{\infty}P_{n,r}^{\prime}(0)=\left\{\begin{array}{l}
\displaystyle -\lambda 1+\lambda 1=0, \quad (n,r) \in \left\{0,1,\ldots, m-1\right\}\times  \mathbb{Z}_+=\mathbb{F},\\  \\
\displaystyle -(\lambda+\mu) 1+\lambda 1+ \mu 1=0, \quad (n,r) \in \mathbb{Z}_+^2 \setminus \mathbb{F}.
\end{array}
\right.
\]
Since the functions $P_{n,r}\in \mathcal{H}(\mathbb{C})$ we have that for $t\geq 0$:
\[
P_{n,r}(t)=\sum_{k=0}^{\infty} \frac{P_{n,r}^{(k)}(0)}{k!}t^k \implies \sum_{r=0}^{\infty}P_{n,r}(t)= \sum_{r=0}^{\infty}P_{n,r}(0)+ \sum_{r=0}^{\infty}\sum_{k=1}^{\infty} \frac{P_{n,r}^{(k)}(0)}{k!}t^k.
\]
Taking into account (\ref{bobo}) we obtain 
\[
\left|P_{n,r}^{(k)}(0)\right|\leq 2(\mu+\lambda) \max_{r\in \mathbb{Z}_+} \left|P_{n,r}^{(k-1)}(0)\right| \leq 2^2(\mu+\lambda)^2 \max_{r\in \mathbb{Z}_+} \left|P_{n,r}^{(k-2)}(0)\right|
\]
\[
 \leq \cdots \leq 2^k(\mu+\lambda)^k \max_{r\in \mathbb{Z}_+} \left|P_{n,r}(0)\right|=2^k(\mu+\lambda)^k.
\]
In the last equality we took into account the initial condition in the problem (\ref{initialvalueproblem}). We fixed an arbitrary  compact interval $\Delta \subset \mathbb{R}_+$, and any ordered pair $(n,k)\in \mathbb{Z}_+^2$  we have that the sequence of functions $\displaystyle \left\{\frac{P_{n,r}^{(k)}(0)t^k}{k!}\right\}_{r\in \mathbb{Z}_+}$ is absolutely bounded in $\Delta$. Then in this interval we can apply Fubini's  Theorem, which completes the proof
\[
\sum_{r=0}^{\infty}P_{n,r}(t)= \sum_{r=0}^{\infty}P_{n,r}(0)+ \sum_{k=1}^{\infty}  \sum_{r=0}^{\infty}\frac{P_{n,r}^{(k)}(0)}{r!}t^k\equiv 1.
\]
\end{proof}

Recall the sequence of distribution forms ${\bf \mathcal{L}}=\left(\mathcal{L}_0,\mathcal{L}_1, \mathcal{L}_2, \ldots\right)$. Its components also satisfy the recurrence relation (\ref{preccurencelatin}) replacing ${\bf q}_r$ by $\mathcal{L}_r$  . This means that ${\bf \mathcal{L}}(x){\bf K}=(x+\lambda+\mu){\bf \mathcal{L}}(x)$, which implies that  ${\bf \mathcal{L}}(x){\bf K}^k=(x+\lambda+\mu)^k{\bf \mathcal{L}}(x),$ $k\in \mathbb{Z}_+$.

The third property is deduced by the expression  (\ref{Abbey}) using Lemma \ref{identityrecurrenceorthgonality}.

\begin{proposition}\label{product} Let $P_{n,r}(t)$  be the function defined by the integral forms in (\ref{Abbey}). They satisfy the following relations
\[
P_{n,r}(s+t)=\sum_{k=0}^{\infty}P_{n,k}(s)P_{k,r}(t), \quad (s,t)\in \mathbb{R}_+^2, \quad (n,r)\in \mathbb{Z}_+^2.
\] 
\end{proposition}

\begin{proof}In fact from (\ref{Abbey})
\[
P_{n,k}(s)P_{k,r}(t)=\int e^{ x_1 s}\, Q_n(x_1)\,  \mathcal{L}_k(x_1)\, d x_1\, \int e^{ x_2 t}\, Q_k(x_2) \mathcal{L}_r(x_2)\, d x_2
\]
\[
=e^{-(\mu+\lambda)(s+t)}
\]
\[
\times \int e^{ (x_1+\mu_1+\lambda) s}\, Q_n(x_1)\,  \mathcal{L}_k(x_1) \, d x_1\, \int e^{ (x_2+\mu+\lambda) t}\, Q_k(x_2)  \mathcal{L}_r(x_2)\, d x_2
\]
\[
=e^{-(\mu+\lambda)(s+t)} \sum_{\ell_1=0}^{\infty} \frac{t^{\ell_1}}{\ell_1!} \sum_{\ell_2=0}^{\infty} \frac{s^{\ell_2}}{\ell_2!}  \int  (x_1+\mu+\lambda)^{\ell_1} Q_n(x_1) \mathcal{L}_k(x_1)\, d x_1
\]
\[
\times \int  (x_2+\mu+\lambda)^{\ell_2} Q_k(x_2) \mathcal{L}_r(x_2)\, d x_2.
\]
We obtain that
\[
P_{n,k}(s)P_{k,r}(t)=e^{-(\mu+\lambda)(s+t)} \sum_{\ell_1=0}^{\infty} \frac{t^{\ell_1}}{\ell_1!}  \sum_{\ell_2=0}^{\infty} \frac{s^{\ell_2}}{\ell_2!} 
\]
\[
\times {\bf e}_n^{\top} {\bf K}^{\ell_1} \int {\bf Q}(x_1){\bf \mathcal{L}}(x_1) d x_1 {\bf e}_k {\bf e}_k^{\top} \int {\bf Q}(x_2){\bf \mathcal{L}}(x_2) d x_2 {\bf K}^{\ell_2} {\bf e}_r 
\]
\[
=e^{-(\mu+\lambda)(s+t)} \sum_{\ell_1=0}^{\infty} \frac{t^{\ell_1}}{\ell_1!}  \sum_{\ell_2=0}^{\infty} \frac{s^{\ell_2}}{\ell_2!} 
 {\bf e}_n^{\top} {\bf K}^{\ell_1} {\bf e}_k {\bf e}_k^{\top} {\bf K}^{\ell_2} {\bf e}_r.
\]
We have taken into account that $\displaystyle \int {\bf Q}(x){\bf \mathcal{L}}(x) d x=\mathbb{I}$, which is the matrix form of the orthogonality conditions in (\ref{orthogonalityQ}). Then
\[
\sum_{k=0}^{\infty}P_{n,k}(s)P_{k,r}(t)=e^{-(\mu+\lambda)(s+t)} \sum_{\ell_1=0}^{\infty} \frac{t^{\ell_1}}{\ell_1!}  \sum_{\ell_2=0}^{\infty} \frac{s^{\ell_2}}{\ell_2!} 
 {\bf e}_n^{\top} {\bf K}^{\ell_1} \sum_{k+0}^{\infty} {\bf e}_k {\bf e}_k^{\top} {\bf K}^{\ell_2} {\bf e}_r
 \]
 Observe that $\displaystyle \mathbb{I}=\sum_{k+0}^{\infty} {\bf e}_k {\bf e}_k^{\top}$. This implies 
 \[
\sum_{k=0}^{\infty}P_{n,k}(s)P_{k,r}(t)=e^{-(\mu+\lambda)(s+t)} \sum_{\ell_1=0}^{\infty} \frac{t^{\ell_1}}{\ell_1!}  \sum_{\ell_2=0}^{\infty} \frac{s^{\ell_2}}{\ell_2!} 
 {\bf e}_n^{\top} {\bf K}^{\ell_1}  \int {\bf Q}(x) {\bf \mathcal{L}}(x) \, d x \, {\bf K}^{\ell_2} {\bf e}_r
 \]
\[
=e^{-(\mu+\lambda)(s+t)} \int \sum_{\ell_1=0}^{\infty} \frac{(x+\mu+\lambda)^{\ell_1}t^{\ell_1}}{\ell_1!}  \sum_{\ell_2=0}^{\infty} \frac{(x+\mu+\lambda)^{\ell_2} s^{\ell_2}}{\ell_2!} Q_n(x) \mathcal{L}_r(x) \, d x
\]
\[
=e^{-(\mu+\lambda)(s+t)} \int e^{(x+\mu+\lambda)s} e^{(x+\mu+\lambda)t} Q_n(x) \mathcal{L}_r(x)  \, d x
\]
\[
=e^{-(\mu+\lambda)(s+t)} \int e^{(x+\mu+\lambda)(s+t)} Q_n(x) \mathcal{L}_r(x) \, d x
\]
\[
= \int e^{x(s+t)} Q_n(x) \mathcal{L}_r(x)  \, d x=P_{n,r}(s+t),
\]
which completes the proof.
\end{proof}

We now show the last property

\begin{proposition}\label{limitP}  Let $P_{n,r}(t)$  be the function defined by the integral forms in (\ref{Abbey}). It satisfies:
\[
\displaystyle \lim_{t\to 0}P_{n,r}(t)=\delta_{n,r}, \qquad (n,r)\in \mathbb{Z}_+^2.
\]
\end{proposition}

\begin{proof} From (\ref{Abbey}) 
\[
\lim_{t \to 0}P_{n,r}(t)=\int \lim_{t \to 0} e^{ xt}\, Q_n(x)\, \mathcal{L}_r(x)\, d x= \int  Q_n(x)\, \mathcal{L}_r(x)\, d x=\delta_{n,r}.
\]
In the last equality we used the identities in (\ref{orthogonalityQ}). The proof is completed.
\end{proof}

In the proof of Proposition \ref{property1} we have already shown that $\mathbb{Z}_+$ is irreductible for $P_{n,r}$. We now prove that $P_{n,r}$ is transient and its steady state

\begin{proposition}\label{steady}  Let $P_{n,r}(t)$  be the function defined by the integral forms in (\ref{Abbey}). Tthe following two properties hold:
\begin{enumerate}
\item[]
\item $\displaystyle \lim_{t\to \infty}P_{n,r}(t)=0,$ $(n,r)\in \mathbb{Z}_+^2$,
\item[]
\item $\displaystyle \int_0^{\infty}P_{n,n}(t) <\infty$, $n\in \mathbb{Z}_+$.
\end{enumerate}
\end{proposition}

\begin{proof} In order to prove the condition (1) we rewrite (\ref{Abbey}) as follows
\[
P_{n,r}(t)=\int e^{x t} Q_n(x) \sum_{j=0}^{m-1} q_{j,r}(x)\rho_j(x) \, \rho(x)\,d x
\]
\[
+\frac{(-1)^{j-1}}{\lambda^j(j-1)!} \left[ \left.\frac{\partial^{j-1}}{\partial x^{j-1}}\left(e^{x t} Q_n(x) \sum_{j=0}^{m-1} q_{j,r}(x)\right)\right\vert_{x=-\lambda}\right.
\]
\[
\left.-\left.\frac{\partial^{j-1}}{\partial x^{j-1}}\left(e^{x t} Q_n(x) \sum_{j=0}^{m-1} q_{j,r}(x)\right)\right\vert_{x=-\lambda-\mu} \right].
\]

First we focus on the two last terms. They contain the $(j-1)$th derivatives of products where one of their factors are exponentials that go to zero as $t$ increases to $\infty$. After differentiation these exponential remain. Theother factors have bounded derivatives. Then these terms decreases to zero as $t \to \infty$. In the first term, since the only factor in the integrand that dependence on $t$ is the exponential $e^{x t}$, a sufficient condition to show that $P_{n,r}(t)$ tends to zero as $t \to \infty$, is that  every $z \in \Sigma_0$ has a nonpositive real part $\Re z \leq 0$. This implies that we need to prove that
\begin{equation}\label{negativecond}
\frac{m+1}{m}\left(\frac{\mu \lambda^m}{m}\right)^{1/(m+1)} \leq \lambda+\mu,
\end{equation}
 which is equivalent to
 \[
\left(1+\frac{1}{m}\right)^{m+1} \frac{1}{m} \leq \left(1+\frac{\mu}{\lambda}\right)^{m+1} \frac{\lambda}{\mu}. 
 \]
 Let us study the function  $\displaystyle g(x)=\frac{(1+x)^{m+1}}{x}, \quad x>0$, and observe that 
 \[
\lim_{x\to 0_+} g(x)=\lim_{x\to \infty} g(x)=\infty.
 \]
 It also has a unique critical point in $\mathbb{R}_+\setminus \{0\}$ at $x=1/m$. The function $g$ gets an absolute minimum at this point $x=1/m$. Then we compare 
 \begin{equation}\label{m2}
\frac{g(1/m)}{\displaystyle \left(1+\frac{1}{m}\right)^{m+1} \frac{1}{m}}=\frac{\displaystyle \left(1+\frac{1}{m}\right)^{m+1} m}{\displaystyle \left(1+\frac{1}{m}\right)^{m+1} \frac{1}{m}}=m^2\geq 1.
 \end{equation}
 This concludes that (\ref{negativecond}) holds no matter what $\lambda$ and $\mu$ positive values have, which completes the proof of condition (1) which means that $\mathbb{Z}_+$ is irreductible with respect to $P_{n,r},$ $(n,r)\in \mathbb{Z}_+^2$.

Let us prove the condition (2).  Using Fubini's theorem we have that
\[
\displaystyle \int_0^{\infty}P_{n,r}(t)=\int_0^{\infty} \int e^{x t} Q_n(x) \sum_{j=0}^{m-1} q_{j,r}(x) \, d \, \sigma_j(x) \, d \, t
\]
\[
= \int \left(\int_0^{\infty} e^{x t} x d\, t \right) \frac{1}{x} Q_n(x) \sum_{j=0}^{m-1} q_{j,r}(x) \, d \, \sigma_j(x)= \int \frac{1}{x} Q_n(x) \sum_{j=0}^{m-1} q_{j,r}(x) \, d \, \sigma_j(x).
\]
Observe from (\ref{m2}) that if $m>1$, the set $S_0$ does not contain the origin $z=0$, hence for each $j\in \{0,1,2,,\ldots, m-1\}$ the rational function $\displaystyle \frac{1}{x} Q_n(x) q_{j,r}(x)$ is integrable with respect to the distribution $\sigma_j$. This completes the proof. 
\end{proof}

Finally combining Propositions \ref{unireal}, \ref{property1},  \ref{honesty}, \ref{product}, \ref{limitP}, and \ref{steady} we complete the proof of Theorem \ref{main}.

\end{document}